\documentclass[12pt, A4paper]{amsart}
\usepackage{amsmath,amssymb,amscd,amsthm, euscript}
\usepackage[matrix,arrow,curve]{xy}
\newtheorem{theorem}{Theorem}[section]
\newtheorem{predl}[theorem]{Proposition}
\newtheorem{lemma}[theorem]{Lemma}
\newtheorem{corollary}[theorem]{Corollary}

\theoremstyle{definition}
\newtheorem{definition}[theorem]{Definition}

\theoremstyle{remark}
\newtheorem{example}[theorem]{Example}
\newtheorem{remark}[theorem]{Remark}

\newcommand{\quot }{/\!\!/}

\newcommand{\N}{\mathbb N}

\newcommand{\G}{\mathbb G}
\newcommand{\A}{\mathbb A}

\newcommand{\BB}{\mathcal B}
\newcommand{\CC}{\mathcal C}

\newcommand{\D}{\mathcal D}

\newcommand{\FF}{\mathcal F}

\newcommand{\LL}{\mathcal L}

\newcommand{\TT}{\mathsf T}
\renewcommand{\O}{\mathcal O}

\renewcommand{\k}{\mathsf k}
\newcommand{\Mod}{\mathrm{{-}Mod}}
\newcommand{\mmod}{\mathrm{{-}mod}}

\newcommand{\coh}{\mathrm{coh}}
\newcommand{\qcoh}{\mathrm{qcoh}}
\newcommand{\perf}{\mathrm{perf}}
\newcommand{\Id}{\mathrm{Id}}
\newcommand{\Kern}{\mathrm{Kern}}
\newcommand{\Augm}{\mathrm{Sk}_{\ge 0}}

\newcommand{\ra}{\mathbin{\rightarrow}}

\newcommand{\xra}{\xrightarrow}

\renewcommand{\ge}{\geqslant}
\renewcommand{\~}{\tilde }

\newcommand{\bul}{\bullet}

\DeclareMathOperator{\Hom}{\textup{Hom}}

\DeclareMathOperator{\Ext}{\textup{Ext}}

 \DeclareMathOperator{\Pic}{\mathrm{Pic}}

\DeclareMathOperator{\EEnd}{\mathcal{E}\mathrm{nd}}
 \DeclareMathOperator{\Spec}{\mathrm{Spec}}
\DeclareMathOperator{\har}{\mathrm{char}} \DeclareMathOperator{\Ob}{\mathrm{Ob}}
\DeclareMathOperator{\Ker}{\mathrm{ker}} \DeclareMathOperator{\Repr}{\mathrm{Rep}}
\DeclareMathOperator{\repr}{\mathrm{rep}} 

\def\a{\alpha}

\newcommand{\e}{\varepsilon}
\newcommand{\s}{\sigma}

\begin{document}

\author{Alexey ELAGIN}
\address{Institute for Information Transmission Problems RAS (Kharkevich Institute), Moscow, and Laboratory of Algebraic Geometry, SU-HSE, Moscow}
\email{alexelagin@rambler.ru}
\title[Cohomological descent theory]{Cohomological descent theory for a morphism of stacks and for equivariant derived categories}
\thanks{
This work was partially supported by
RFBR grants {(projects 09-01-00242-a, 10-01-93110-NCNIL-a and 10-01-93113-NCNIL-a)}, grants of the President of Russia MD-2712.2009.1 and NSh-4713.2010.1,
"Dynasty" foundation, and AG Laboratory HSE (RF government
grant, ag. 11.G34.31.0023). The author is grateful
to Science Foundation of the SU-HSE for supporting project 10-09-0015.}
\date{}
\maketitle

\begin{abstract}
In the paper we answer the following question: for a morphism of varieties (or, more generally, stacks), when the derived category of the base can be recovered from the derived category of the covering variety by means of descent theory?
As a corollary, we show that for an action of a reductive group on a scheme, the derived category of equivariant sheaves is equivalent to the category of objects, equipped with an action of the group, in the ordinary derived category. \end{abstract}

\section{Introduction}

It is known that sheaves on a variety can be defined locally.
That is, let $S=\bigcup U_i$ be an open covering. To give a sheaf on $S$
is the same as to give a family of sheaves $F_i$ on $U_i$ and a family of isomorphisms $\phi_{ij}\colon F_i|_{U_i\cap U_j}\ra F_j|_{U_i\cap U_j}$ satisfying cocycle conditions: on intersections $U_i\cap U_j\cap U_k$ one has
$\phi_{ik}=\phi_{jk}\circ \phi_{ij}$. Sheaves can be also defined using more general coverings. Let $p\colon X\ra S$  be a covering (for example, a covering of topological spaces or a flat finite morphism of schemes), let $p_i$ and $p_{ij}$ be
the projections of the fibred products $X\times_SX$ and $X\times_SX\times_SX$ onto factors. Then giving a sheaf on the base $S$ is equivalent to giving
a sheaf $F$ on $X$ together with gluing isomorphism
$\theta \colon p_1^*F\ra p_2^*F$ on $X\times_SX$ satisfying the following cocycle condition: isomorphisms $p_{13}^*\theta$ and $p_{23}^*\theta\circ p_{12}^*\theta$ on  $X\times_SX\times_SX$ are equal.

We remark that the first statement above concerning an open covering follows from
the second one for $X=\bigsqcup U_i$.

The natural question is: would similar facts hold for sheaves replaced by arbitrary objects of derived category of coherent sheaves on an algebraic variety?

There is an important special version of this question. Suppose that $G$ is an algebraic group acting on an algebraic variety $X$. By definition, an equivariant sheaf on $X$ is a sheaf $F$ equipped with an action of the group (for finite $G$ an action is given by isomorphisms $\theta_g\colon F\ra g^*F$ for any $g\in G$, which are compatible in the following sense:
$g^*\theta_h\circ \theta_g=\theta_{hg}$ for any pair $g,h\in G$).
For a given action, equivariant quasi-coherent sheaves form an abelian category $\qcoh^G(X)$. Suppose $\FF^{\bul}$ is an object of its derived category
$\D^G(X)=\D(\qcoh^G(X))$, then $G$ acts on the complex $F^{\bul}\in \D(X)$
which is $\FF^{\bul}$ with forgotten group action. For finite $G$
an action of $G$ on $F^{\bul}$ is given in a similar way: as a compatible family of isomorphisms
$\theta_g\colon F^{\bul}\ra g^*F^{\bul}$ in the category $\D(X)$. Is the converse true: given a complex $F^{\bul}$ and an action of an algebraic group on $F^{\bul}$,
do they
define an object in the derived category of equivariant coherent sheaves?

In this paper we answer both above questions. Note that the second question is essentially a special case of the first one if one considers the covering $X\to X\quot G$ where $X$ is a variety and $X\quot G$ is the quotient stack of $X$ by the action of
the group $G$. Hence it is reasonable to work in the category of stacks, not schemes.

To be more precise, our goal is to determine when two categories are equivalent.
These categories are: the derived category of sheaves on the base $S$ and a certain descent category $\D(X)/p$
associated with the covering $X\to S$.
The standard way to define the descent category was described above. An object of the descent category is an object  $F$ in $\D(X)$ equipped with a gluing isomorphism $p_1^*F\ra p_2^*F$
on $X\times_SX$ satisfying the cocycle condition.
We give the criterion in Section~\ref{s7}:

\smallskip
{\bf Theorem~\ref{th_descentforstacks}.} For a flat morphism of stacks $p\colon X\ra S$
the unbounded derived category $\D(S)$ is equivalent to the descent category $\D(X)/p$ associated with $p$ if and only if the natural morphism
$\O_S\ra Rp_*\O_X$ is an embedding of a direct summand.

\smallskip
For comparison of equivariant and non-equivariant derived categories we get a corollary:

\smallskip
{\bf Theorem~\ref{th_descentforequiv}.} For an action of a linearly reductive group $G$
(i.e., a group with semisimple category of representations) on a scheme $X$
the derived category of equivariant sheaves $\D^G(X)$ is equivalent to the descent category $\D(X)^G$, formed by objects of $\D(X)$ equipped with action of~$G$.

\smallskip
Often it is convenient to define descent data in the different way, the one coming from monad theory.
Comonad descent category is the category of comodules over a certain comonad on the category $\D(X)$. This comonad is the one associated with the adjoint pair of functors $p^*$ and $p_*$ between unbounded derived categories
$\D(X)$ and $\D(S)$. We prove that these two descent categories are equivalent if the morphism $p\colon X\to S$ is flat (Proposition~\ref{prop_twotypes}).
Thus the language of comonads can be used. With its help we prove
Theorem~\ref{th_descentforstacks}, basing on classical Beck theorem.

The paper is organized as follows.
Results of sections 2-6 are mostly well-known. In Section 2 facts concerning cosimplicial categories and related descent categories are collected. In Section 3 we present, following Barr-Wells~\cite{TTT}, comonad theory: definitions, Comparison Theorem and criteria of comparison functor being an equivalence.
New material here is the special case of this criterion for triangulated categories. In Section 4 we prove that the classical way of giving descent data is equivalent to the one related with comodules over comonad.
In Section 5 we collect facts about subcategories in descent categories, which are
needed for dealing with bounded derived categories and categories of perfect complexes.
In Section 6 we present basics on derived categories of sheaves on stacks.
Section 7 is the central one in the paper. In it, basing on results of Section 3, we find out when
the derived category of the base can be recovered from the derived category of the covering stack as a descent category.
In Section 8 we introduce and study SCDT property  for a morphism of stacks. We show that finite flat morphisms and smooth projective morphisms in zero characteristic are SCDT.
In Section 9 we apply general results on descent for stacks in the special case
of equivariant derived categories.


Author thanks D.\,Orlov for constant attention to the work and inspiring discussions and A.\,Kuznetsov for his valuable remarks, corrections, support, and for careful reading of this text.
Also author is grateful to Kyoto University for the invitation to a winter school on algebraic geometry in January 2009, where the main part of this paper was written, and for their hospitality.

\section{Cosimplicial constructions}\label{section_cosimplicial}

Let $\Delta_0$ be the category, whose objects are sets of the form  $[1,\ldots,n], n\in \N$, and an empty set, and whose morphisms are non-decreasing maps between them. Let $\Delta\subset\Delta_0$ be its full subcategory formed by non-empty sets.

By definition, a cosimplicial object of a category $\CC$ (for example, a cosimplicial set, a cosimplicial scheme) is a functor from $\Delta$ to~$\CC$. Taking the 2-category of categories $\mathrm{Cats}$ for $\CC$, we obtain a definition of a cosimplicial category.

\begin{definition}
A cosimplicial category is a covariant 2-functor from $\Delta$  to 2-category $\mathrm{Cats}$ of categories and functors; an augmented cosimplicial category is a covariant 2-functor $\Delta_0 \ra \mathrm{Cats}$. More explicitly, a cosimplicial
category $\CC_{\bul}$ (resp. augmented cosimplicial category $\CC_{\bul}$)
consists of the following data:
\begin{enumerate}
\item a set of categories $\CC_k$, $k=0,1,2\ldots$ (resp. $k=-1,0,1,2\ldots$) indexed by objects of~$\Delta$ (resp. $\Delta_0$; here $\CC_k$ is the category associated with the set $[1,\dots, k+1]\in\Delta$, $\CC_{-1}$ corresponds to $\emptyset$);
\item a set of functors $P_f^*\colon\CC_m\ra \CC_n$ indexed by morphisms of
$\Delta$ (of $\Delta_0$), i.e. by nondecreasing maps
$f\colon [1,\ldots,m+1]\ra [1,\ldots,n+1]$;
\item a set of functor isomorphisms $\epsilon_{f,g}\colon P_f^*P_g^*\ra P_{fg}^*$ indexed by composable pairs $f,g$ of maps.
\end{enumerate}
Isomorphisms in 3) should obey the following cocycle condition: diagram
$$\xymatrix{P_f^*P_g^*P_h^* \ar[r]^{\epsilon_{f,g}}\ar[d]^{\epsilon_{g,h}} & P_{fg}^*P_h^*\ar[d]^{\epsilon_{fg,h}}\\ P_f^*P_{gh}^* \ar[r]^{\epsilon_{f,gh}} & P_{fgh}^*}$$
is commutative for any composable triple $f,g,h$ of maps.

Simplicial and augmented simplicial category are defined as contravariant 2-functors $\Delta \ra \mathrm{Cats}$ and $\Delta_0 \ra \mathrm{Cats}$.
\end{definition}

For a given augmented cosimplicial category $\CC_{\bul}=[\CC_{-1},\CC_0,\CC_1,\ldots,P_f^*]$, an \emph{augmentation} is
formed by the category $\CC_{-1}$ and functors  $P_f^*$ that are defined on $\CC_{-1}$.
Removing the augmentation from $\CC_{\bul}$ we obtain the cosimplicial category $[\CC_0,\CC_1,\ldots,P_f^*]$, which is denoted by $\Augm(\CC_{\bul})$.

Using terminology of~\cite[19.1]{KSh} one can say that a cosimplicial category is a prestack on $\Delta$.

\begin{example}
\label{example_standard}
Let $X\ra S$ be a morphism of schemes. Then schemes $S,X,X\times_S X,X\times_S X\times_S X,\ldots$ and morphisms
$$p_f\colon \underbrace{X\times_SX\times\ldots\times X}_{n}\ra \underbrace{X\times_SX\times\ldots\times X}_{m}$$ between them given by the rule
$$p_f(x_1,\ldots,x_n)=(x_{f(1)},\ldots,x_{f(m)})$$
for $f\in \Hom_{\Delta_0}([1,\ldots,m],[1,\ldots,n])$
form an augmented simplicial scheme. Categories of sheaves on these schemes and pull-back functors between these categories give an important example of an augmented cosimplicial category.
\end{example}

It is convenient to think about categories $\CC_{-1},\CC_0, \CC_1,\ldots$ as about categories of sheaves on $S,X,X\times_S X,X\times_S X\times_S X,\ldots$.
The notation for functors $P_f^*$ we are using is a reminiscence of
pullback functors on categories of sheaves.
For instance, we will denote a functor $P_{f}^*\colon \CC_1\ra
\CC_2$, where $f\colon [1,2]\ra [1,2,3]$ is a map such that $f(1)=1$, $f(2)=3$, by
$P_{13}^*$. It may be thought of as the pull-back functor for a projection
$p_{13}\colon X\times_SX\times_SX\ra
X\times_SX$. We denote a functor $P_f^*\colon\CC_{-1}\ra\CC_0$ corresponding to the only map $\emptyset\ra [1]$ by $P^*$, it plays the role of the pull-back functor for the morphism $p\colon X\ra S$. We denote by $D^*$ the functor $P_f^*\colon \CC_1\ra \CC_0$ for the only map $f\colon [1,2]\ra [1]$. This functor may be thought of as induced by a diagonal embedding
$d\colon X\ra X \times_SX$.

For any cosimplicial category $\CC_{\bul}=[\CC_{0}, \CC_1,
\ldots, p_{\bul}^*]$ there is a well-defined descent category denoted by $\Kern(\CC_{\bul})$,
see.~\cite[19.3]{KSh}.

\begin{definition}[{classical descent category}]
\label{def_dxbar} Objects of $\Kern(\CC_{\bul})$ are pairs
$(F,\theta)$ where  $F\in \Ob  \CC_0$ and $\theta$ is an isomorphism $P_1^*F\ra
P_2^*F$ satisfying the following cocycle condition: the diagram
$$\xymatrix{&P_{13}^*P_1^*F \ar[r]^{P_{13}^*\theta} \ar@{-}[ld]_{\sim} &
P_{13}^*P_2^*F \ar@{-}[rd]^{\sim}&\\
P_{12}^*P_1^*F \ar[rd]_{P_{12}^*\theta} &&& P_{23}^*P_2^*F. \\
&P_{12}^*P_2^*F \ar@{-}[r]_{\sim} & P_{23}^*P_1^*F \ar[ru]_{P_{23}^*\theta}}$$
is commutative. Functor isomorphisms from the definition of a cosimplicial
category are denoted here by lines with the ``$\sim$'' sign.
A morphism in $\Kern(\CC_{\bul})$ from $(F_1,\theta_1)$ to
$(F_2,\theta_2)$ is a morphism $f\in \Hom_{\CC_0}(F_1,F_2)$ such that
$P_2^*f\circ \theta_1=\theta_2\circ P_1^*f$.
\end{definition}
\begin{remark}
Skipping canonical isomorphisms, the cocycle condition is usually written as
$$P_{23}^*\theta\circ P_{12}^*\theta=P_{13}^*\theta.$$
\end{remark}
Category $\Kern(\CC_{\bul})$ can also be defined in the other way.
\begin{definition}
\label{def_dxbar1} An object of $\Kern(\CC_{\bul})$ consists of a family of objects
$F_i\in \CC_i$, $i=0,1,2,\ldots$ and of a  family of isomorphisms
$\phi_f\colon P_f^*F_m\xra{\sim} F_n$ for all maps $f\colon [1,\ldots,m{+}1]\ra [1,\ldots,n+1]$ in $\Delta$ that satisfy equalities
$\phi_{gf}=\phi_{g}\circ P_{g}^*\phi_{f}$ for all composable pairs $(f,g)$.
A morphism in $\Kern(\CC_{\bul})$ from $(F_{\bul},\phi_{\bul})$ to
$(F'_{\bul},\phi'_{\bul})$ is a family of morphisms $\rho_i\colon
F_i\ra F'_i$ compatible with $\phi_{\bul}$ and $\phi_{\bul}'$.
\end{definition}

The following fact is well-known.
\begin{predl}
\label{prop_kernkern}
Definitions ~\ref{def_dxbar} and~\ref{def_dxbar1} are equivalent.
\end{predl}


\begin{proof}
We give a sketch of the proof for the convenience of the reader.

Consider an object of the category from Definition~\ref{def_dxbar1}, which consists of two families $F_{\bul}=(F_i)$ and $\phi_{\bul}=(\phi_f)$. Construct an object of the category from Definition~\ref{def_dxbar}:
let $F=F_0$ and  define $\theta\colon P_1^*F\ra P_2^*F$ to be the composition
$$\phi_{i_2}^{-1}\phi_{i_1}\colon  P_{i_1}^*F_0\xra{\phi_{i_1}} F_1 \xra{\phi_{i_2}^{-1}} P_{i_2}^*F_0.$$
Here $i_1$ and $i_2$ denote two morphisms $[1]\ra [1,2]$  in $\Delta$,
$i_1(1)=1, i_2(1)=2$. The compatibility condition for $\phi_f$ implies the cocycle condition for $\theta$.

Conversely, consider an object $(F,\theta)$ of the category from
Definition~\ref{def_dxbar}.  Let us construct an object of the category from Definition~\ref{def_dxbar}. Let
$f_i\colon [1]\ra [1,\ldots, i+1]$ be the map sending $1$ to $1$.
Take $F_i=P_{f_i}^*F$. For a map $f\colon [1,\ldots, m+1]\ra [1,\ldots,n+1]$ define an isomorphism $\phi_f$ as follows. Suppose $f(1)=r$, consider a map $g\colon [1,2]\ra[1,\ldots,n+1]$ such that $g(1)=1, g(2)=r$.
Take $\phi_f$ equal to the composition of isomorphisms
\begin{multline*}
P_f^*F_m=P_f^*P_{f_m}^*F\xra{\epsilon_{f,f_m}} P_{ff_m}^*F=P_{gi_2}^*F \xra{\epsilon_{g,i_2}^{-1}} \\
\ra P_g^*P_{i_2}^*F \xra{P_g^*\theta^{-1}} P_g^*P_{i_1}^*F
\xra{\epsilon_{g,i_1}} P_{gi_1}^*F = P_{f_n}^*F=F_n.
\end{multline*}
It is easy to see that the cocycle condition for $\theta$ implies that $\phi_f$ are compatible.
\end{proof}

For a category $\Kern(\CC_{\bul})$ constructed from a cosimplicial category
$\CC_{\bul}$, there are forgetful functors
$\Kern(\CC_{\bul})\ra \CC_k$ defined in an obvious way.
It turns out that these functors extend $\CC_{\bul}$ into an augmented cosimplicial category.
Put $\CC_{-1}=\Kern(\CC_{\bul})$. For a morphism $f\colon\emptyset\ra [0,\ldots,n]$, define the functor
$P_f^*\colon \CC_{-1}\ra \CC_n$
as forgetting: $(F_{\bul},\phi_{\bul})\mapsto F_n$ on objects, $\rho_{\bul}\mapsto \rho_n$ on morphism
(we use Definition~\ref{def_dxbar1} here).
\begin{predl}
\label{prop_comparisonsimpl}
\begin{enumerate}
\item Categories $\CC_{-1}=\Kern(\CC_{\bul}), \CC_0,\CC_1,\ldots$ and functors $P_f^*$ between them form an augmented cosimplicial category $\~{\CC}_{\bul}$.
\item
The category $\~{\CC}_{\bul}$ has the following universal property: for any extension of $\CC_{\bul}$ to an  augmented cosimplicial category
$\~{\CC}'_{\bul}=[\CC'_{-1},\CC_0,\CC_1,\ldots,P_f'^*]$ there exists unique (up to an isomorphism) functor $\Phi\colon \CC'_{-1}\ra \Kern(\CC_{\bul})=\CC_{-1}$
that gives, together with a family of identity functors, a functor
$\~\CC'_{\bul}\ra \~{\CC}_{\bul}$ between augmented cosimplicial categories.
\end{enumerate}
\end{predl}
\begin{remark}
Part 2 of the above proposition is an analogue of comparison theorem~\ref{th_comparison} for comodules over a comonad, see below.
\end{remark}

\begin{proof}
1. Given a pair of maps $f\colon \emptyset\ra [1,\ldots,m+1]$, $g\colon [1,\ldots,m+1]\ra [1,\ldots,n+1]$ define the functor isomorphism
$\epsilon_{g,f}\colon P_g^*P_f^*\xra{\sim}P_{gf}^*$ as follows: for an object
$(F_{\bul},\phi_{\bul})\in\Kern(\CC_{\bul})$ put $$P_g^*P_f^*((F_{\bul},\phi_{\bul}))=P_g^*F_m\xra{\phi_g} F_n=
P_{gf}^*((F_{\bul},\phi_{\bul})).$$
The compatibility condition for $\phi$ implies the cocycle condition for these ``new'' $\epsilon$.

2. Suppose $H$ is an object of $\CC'_{-1}$. Define the functor $\Phi$ on $H$ as the pair $(F_{\bul},\phi_{\bul})$, where
$$F_k=P_{f_k}'^*H\in\~\CC'_k=\CC_k$$
and $f_k\colon \emptyset\ra [1,\ldots,k+1]$ is the only map. Define $\phi_f$
for $f\colon [1,\ldots,m+1]\ra [1,\ldots,n+1]$ by
$$P_f^*F_m=P_f^*P_{f_m}'^*H=P_f'^*P_{f_m}'^*H\xra{\epsilon_{f,f_m}}
P_{ff_m}'^*H=P_{f_n}'^*H=F_n.$$
For a morphism $u\colon H_1\ra H_2$ take $\Phi$ to be $\rho_{\bul}$,
where $\rho_k\colon P_{f_k}'^*H_1\ra P_{f_k}'^*H_2$ is $P_{f_k}'^*u$.
Clearly, the functor $\Phi$ has all necessary properties.
\end{proof}

Mostly, we are interested in augmented cosimplicial categories
$$[\CC_{-1},\CC_{0}, \CC_1, \CC_2,\ldots, P_{\bul}^*]$$
satisfying two additional assumptions A1 and A2.
\begin{equation*}
\text{Functors $P_f^*\colon \CC_m\ra \CC_n$ have right adjoint functors
$P_{f*}\colon\CC_n\ra\CC_m$.} \leqno{(A1)}
\end{equation*}

One can check that functors
$P_{\bul *}$ form augmented simplicial category
$$[\CC_{-1}, \CC_{0}, \CC_1, \ldots, P_{\bul *}].$$
These functors may be thought of as push-forward functors on categories of sheaves.

To state the second assumption, consider a commutative square in the category $\Delta_0$ and the corresponding square of categories and functors:
$$\xymatrix{[1,\ldots,m+n-r+1]&[1,\ldots,n+1] \ar[l]_-{f'}\\
[1,\ldots,m+1] \ar[u]_{g'} & [1,\ldots,r+1], \ar[l]_{f}
\ar[u]_g}\qquad\qquad
\xymatrix{{\CC_{m+n-r}}\ar@<-0.5mm>[rr]_{P_{f'*}}\ar@<-0.5mm>[d]_{P_{g'*}}&&{\CC_n} \ar@<-0.5mm>[ll]_{P_{f'}^*}\ar@<-0.5mm>[d]_{P_{g*}}\\
{\CC_m} \ar@<-0.5mm>[u]_{P_{g'}^*} \ar@<-0.5mm>[rr]_{P_{f*}}&&
{\CC_r}\ar@<-0.5mm>[ll]_{P_f^*} \ar@<-0.5mm>[u]_{P_g^*}.}$$
If maps $f$ and $f'$  (or $g$ and $g'$) are injective and
$[1,\ldots,m+n-r+1]=\mathrm{Im}\ f'\cup \mathrm{Im}\ g'$, then we say that the square is \emph{exact Cartesian}. For arbitrary square two following natural base change morphisms are defined:
\begin{equation}
\label{equation_bc}
P_g^*P_{f*}\ra P_{f'*}P_{g'}^*\quad\text{and}\quad P_f^*P_{g*}\ra P_{g'*}P_{f'}^*.
\end{equation}
For example, a morphism $P_g^*P_{f*}\ra P_{f'*}P_{g'}^*$ can be defined as a composition
$$P_g^*P_{f*}\xra{\eta} P_g^*P_{f*}P_{g'*}P_{g'}^*\xra{\sim}
P_g^*P_{g*}P_{f'*}P_{g'}^*\xra{\e}  P_{f'*}P_{g'}^*$$
or as a composition
$$P_g^*P_{f*}\xra{\eta} P_{f'*}P_{f'}^*P_g^*P_{f*}\xra{\sim}
P_{f'*}P_{g'}^*P_f^*P_{f*}\xra{\e}  P_{f'*}P_{g'}^*,$$ where  $\eta$ and
$\e$ denote canonical adjunction morphisms.
It is  easy to check that these two ways are equivalent.

The second assumption is an axiomatization of the flat base change formula.
\begin{equation*}
\text{Base change morphisms~\eqref{equation_bc} are isomorphisms for any exact Cartesian square.}
\leqno{(A2)}
\end{equation*}

\begin{predl}
\label{prop_augmcosimplcat}
Let $\CC_{\bul}$ be a cosimplicial category, let $\~{\CC}_{\bul}$ be
the augmented cosimplicial category obtained from $\CC_{\bul}$ by adding
$\Kern(\CC_{\bul})$. Suppose $\CC_{\bul}$ satisfies assumptions A1 and A2, then
$\~{\CC}_{\bul}$ also does.
\end{predl}
\begin{proof}
To check A1, we need to prove that for any morphism $f$ in $\Delta_0$ the functor $P_f^*$ in the category $\~{\CC}_{\bul}$ possesses a right adjoint functor. If the morphism $f$ lies in~$\Delta$ then it has an adjoint functor by assumption, therefore we need to consider morphisms of the form $f\colon \emptyset \ra [1,\ldots,n]$. Decompose $f$ into a composition $\emptyset \ra [1]\to [1,\ldots,n]$ to see that it is enough to show the following: the forgetful functor
$P^*\colon \Kern(\CC_{\bul})\ra \CC_0$ has a right adjoint functor.

Define a functor $P_*\colon \CC_0\ra \Kern(\CC_{\bul})$ as follows. For
$F\in \CC_0$ put
$$P_*F=(P_{2*}P_1^*F,\theta_F).$$
Here $\theta_F\colon P_1^*P_2^*P_{1*}F\ra P_2^*P_{2*}P_1^*F$ is the composition of isomorphisms
$$P_1^*P_{2*}P_1^*F\xra{\sim} P_{23*}P_{12}^*P_1^*F\xra{\sim}
P_{23*}P_{13}^*P_1^*F\xra{\sim} P_2^*P_{2*}P_1^*F,$$
where the first and the third isomorphisms are the base changes and the second one is the isomorphism from the definition of a cosimplicial category.
On morphisms we put $P_*f=P_{2*}P_1^*f$. We leave to the reader to check that $\theta_F$ satisfies the cocycle condition.

To see that $P^*$ is adjoint to $P_*$, define adjunction morphisms
$$\eta\colon\Id_{\Kern(\CC_{\bul})}\ra P_*P^*$$ and $$\e\colon P^*P_*\ra
\Id_{\CC_0}.$$

Define $\eta$ on an object $(F,\theta)\in\Kern(\CC_{\bul})$ as a composition $$ \eta_{(F,\theta)}\colon F\xra{\eta_F} P_{2*}P_2^*F\xra{P_2^*\theta^{-1}} P_{2*}P_1^*F.$$
One can check that $\eta_{(F,\theta)}$ is compatible with $\theta$ and $\theta_F$ and hence is a morphism in $\Kern(\CC_{\bul})$.

Define $\e$ on an object $F\in\CC_0$:
$$\e_F\colon P_{2*}P_1^*F\xra{\eta} P_{2*}D_*D^*P_1^*F\xra{\sim} \Id_*\Id^*F=F.$$

The definitions of morphisms $\eta$ and $\e$ and properties of cosimplicial categories imply that both compositions
\begin{gather*}
P_*\xra{\eta P_*} P_*P^*P_* \xra{P_*\e} P_*,\\
P^*\xra{P^*\eta} P^*P_*P^* \xra{\e P^*} P^*
\end{gather*}
are identity. Therefore, the  functors $P^*$ and $P_*$ are adjoint.

\medskip
Recall assumption A2: the base change formula holds for exact Cartesian squares
in~$\Delta_0$.
For squares in  $\Delta$ the base change formula holds by assumption, hence we need to consider squares of the form
$$\xymatrix{[1,\ldots,m+n]&[1,\ldots,n] \ar[l]_-{f'}\\
[1,\ldots,m] \ar[u]_{g'} & \emptyset. \ar[l]_{f} \ar[u]_g}$$
Decomposing $f$ and $g$, we reduce to the case of the following square:
$$\xymatrix{[1,2]&[1] \ar[l]_-{i_2}\\
[1] \ar[u]_{i_1} & \emptyset, \ar[l]_{f} \ar[u]_g}.$$
That is, one has to prove that natural morphisms
$P^*P_*\ra P_{2*}P_1^*$ and $P^*P_*\ra P_{1*}P_2^*$  in $\CC_0$
are isomorphisms.
This is done by a straightforward argument based on definitions.
Proposition~\ref{prop_augmcosimplcat} is proved.
\end{proof}

\section{Comonads and comodules}\label{section_comonad}

We recall some facts from comonad theory. More details can be found in books by
Barr-Wells~\cite[chapter 3]{TTT} and MacLane~\cite[chapter 6]{ML}.

Let $\CC$ be a category.
\begin{definition}
\label{def_comonad} A~\emph{comonad} $\TT=(T,\e,\delta)$ (also the name
\emph{standard construction} is used) on the category
$\CC$ consists of a functor $T\colon \CC\ra \CC$ and of natural transformations of functors $\e\colon T\ra
\Id_{\CC}$ and $\delta\colon T\ra T^2=TT$ such that the following diagrams are commutative:
$$\xymatrix{
T \ar[r]^{\delta} \ar@{=}[rd] \ar[d]^{\delta} & T^2 \ar[d]^{T\e} \\ T^2 \ar[r]^{\e
T} & T, }\qquad \xymatrix{T \ar[r]^{\delta} \ar[d]^{\delta} & T^2 \ar[d]^{T\delta}\\
T^2 \ar[r]^{\delta T} & T^3. }
$$
\end{definition}

\begin{example}
\label{mainexample} Consider a pair of adjoint functors: $P^* \colon\BB\ra \CC$
(left) and $P_*\colon \CC\ra \BB$ (right). Let $\eta\colon \Id_{\BB}\ra P_*P^*$ and
$\e\colon P^*P_*\ra \Id_{\CC}$ be the natural adjunction morphisms. Define a triple $(T,\e,\delta)$ by taking $T=P^*P_*$ and $\delta=P^*\eta P_* \colon P^*P_*\ra P^*P_*P^*P_*$. Then $\TT=(T,\e,\delta)$ is a comonad on $\CC$.
\end{example}

In fact, every comonad can be obtained in this way from a pair of adjoint functors. This follows from the below construction due to Eilenberg-Moore.

\begin{definition}
\label{def_comodule}
  Suppose $\TT=(T,\e,\delta)$ is a comonad on $\CC$.
A \emph{comodule} over $\TT$ (or a \emph{$\TT$-coalgebra}) is a pair
$(F,h)$ where $F\in \Ob  \CC$ and $h\colon F\ra TF$ is a morphism
satisfying the following two conditions:
\begin{enumerate}
\item the composition
$$F\xra{h} TF \xra{\e F} F$$
is the identity;
\item the diagram
$$\xymatrix{
F \ar[r]^h \ar[d]^h  &  TF \ar[d]^{Th} \\
TF \ar[r]^{\delta F} & T^2F }$$ commutes.
\end{enumerate}
A \emph{morphism} between
comodules $(F_1,h_1)$ and $(F_2,h_2)$ is by definition a morphism
$f\colon F_1\ra F_2$ in $\CC$ such that the diagram
$$\xymatrix{F_1 \ar[r]^f \ar[d]_{h_1} & F_2 \ar[d]^{h_2} \\ TF_1 \ar[r]^{Tf} & TF_2}$$
commutes.
\end{definition}

All comodules over a given comonad $\TT$ on $\CC$ form a category
which is denoted $\CC_{\TT}$. Define a functor $Q_*\colon \CC\ra
\CC_{\TT}$ by
$$Q_*F=(TF,\delta F),\qquad Q_*f=Tf,$$
define $Q^*\colon \CC_{\TT}\ra \CC$ to be the forgetful functor: $(F,h)\mapsto F$. Then
the pair of functors $(Q^*,Q_*)$ is an adjoint pair and it generates
the comonad $\TT$ as in Example~\ref{mainexample}.

The category $\CC_{\TT}$ inherits some properties of $\CC$. It is not
hard to verify the below proposition
\begin{predl}
\label{prop_adab} Suppose $\TT=(T,\e,\delta)$ is a comonad on a
category $\CC$. If $\CC$ is additive and $T$ is an additive functor
then $\CC_{\TT}$ is also additive. If $\CC$ is abelian and $T$ is
left exact then $\CC_{\TT}$ is also abelian.
\end{predl}
On the contrary, it is not clear why should the category $\CC_{\TT}$ be
triangulated for a comonad $\TT=(T,\e,\delta)$ with an exact functor
$T$ on a triangulated category $\CC$. One can attempt to define a
triangulated structure on $\CC_{\TT}$ in the following way.
\begin{definition}
\label{quasitriang} Define the shift functor on $\CC_{\TT}$:
$(F,h)[1]=(F[1],h[1])$, $f[1]=f[1]$. Define triangles in $\CC_{\TT}$ to be
diagrams $(F',h')\ra (F,h)\ra (F'',h'')\ra (F',h')[1]$ such that
$F'\ra F\ra F''\ra F'[1]$ is a triangle in $\CC$.
\end{definition}
Unfortunately, taking
cones in $\CC$ is not functorial, therefore without additional
assumptions it is not possible to verify that morphisms in $\CC_{\TT}$ extend into triangles.
But sometimes the above definition does define a triangulated
structure on $\CC_{\TT}$; if it is the case, we will simply say that
$\CC_{\TT}$ is triangulated (keeping in mind that the triangulated
structure is exactly the one defined here). Later
(Proposition~\ref{prop_cttriang}) we shall see that $\CC_{\TT}$ is
triangulated in some special cases: in fact, we will show that
$\CC_{\TT}$ is equivalent (as an abstract category) to some other
triangulated category.

The following result says that the Eilenberg-Moore construction gives a
terminal object among all adjoint pairs producing the same comonad.

\begin{predl}[Comparison theorem, {\cite[3.2.3]{TTT}, \cite[6.3]{ML}}]
\label{th_comparison} Assume that a comonad $\TT=(T,\e,\delta)$ on
a category $\CC$ is defined by an adjoint pair of functors
$P^*\colon \BB\ra \CC, P_*\colon \CC\ra \BB$. Then there exist a
unique (up to an isomorphism) functor (called \emph{comparison
functor}) $\Phi\colon \BB\ra \CC_{\TT}$ such that the diagram of
categories
$$\xymatrix{
&& {\BB} \ar[dd]^{\Phi} \ar@<1mm>[lld]^{P^*}\\
{\CC} \ar@<1mm>[rru]^{P_*} \ar@<-1mm>[rrd]_{Q_*} && \\
&&{\CC_{\TT}} \ar@<-1mm>[llu]_{Q^*} }$$ commutes, i.e. both triangles are commutative:
$$\Phi P_*\cong Q_*,\qquad Q^*\Phi\cong P^*.$$
\end{predl}
\begin{proof}
Define $\Phi\colon \BB\ra \CC_{\TT}$ to be a functor assigning to
an object $H\in \Ob  \BB$ a pair $(P^*H,h)$ where $h\colon P^*H\ra
P^*P_*P^*H$ is $P^*$ applied to the canonical mapping $\eta\colon H\ra P_*P^*H$, and
assigning to a morphism $f$ a morphism $P^*f$. It can be checked in
a straightforward way that $\Phi$ is the required functor.
\end{proof}

We will need criteria for a comparison functor
of being fully faithful and of being an equivalence. Before formulating
these criteria we recall the notions of an equalizer and of a
contractible equalizer.

\begin{definition}
An \emph{equalizer} of a pair of morphisms $d_1,d_2\colon F_1\ra
F_2$ is a morphism $d\colon F\ra F_1$ such that
\begin{enumerate}
  \item $d_1 d=d_2 d$,
  \item for any morphism $d'\colon F'\ra F_1$ such that $d_1 d'=d_2 d'$ there exists a unique morphism $f\colon F'\ra F$ such that $d f=d'$.
\end{enumerate}
\end{definition}
\begin{definition}
\label{contreq} A \emph{contractible equalizer} of a pair of
morphisms $d_1,d_2\colon F_1\ra F_2$ in a category is a morphism
$d\colon F\ra F_1$ for which there exist morphisms $s$ and $t$ as
shown below
$$\xymatrix{F \ar[rr]^d && F_1 \ar@/^2em/[ll]^s \ar@<1mm>[rr]^{d_1} \ar@<-1mm>[rr]_{d_2} && F_2 \ar@/^2em/[ll]^t}$$ satisfying the equalities
\begin{align*}
d_1 d =d_2 d, \quad s d =\Id,\quad t d_1 =\Id, \quad  t d_2 &=d s.
\end{align*}
\end{definition}

Remark that any contractible equalizer is an equalizer and any
equalizer is a monomorphism. Also note that contractible equalizers
are preserved by all functors.

A morphism $f\colon H'\ra H$ in a category $\BB$ is a \emph{split
embedding} or, briefly, \emph{splits} if there exists a left inverse
morphism  $f'$: $f'f=\Id_{H'}$. If the category $\BB$ is additive, it is
also said that $f$ is an embedding of a direct summand.

Recall that a functor $\Phi$ is called \emph{conservative} if for
any morphism $f$ such that $\Phi(f)$ is an isomorphism, $f$ itself
is an isomorphism.

The following theorem is known as Precise Tripleability Theorem, or PTT.
\begin{theorem}[Beck, {\cite[3.14]{TTT},\cite[6.7]{ML}}]In the above notation
\label{th_beck}
\begin{enumerate}
\item The functor $\Phi$ is fully faithful iff for any $H\in \Ob \BB$ the natural
morphism $\eta H\colon H\ra P_*P^*H$ is an equalizer (of some pair).
\item The functor $\Phi$ is an equivalence iff $P^*$ is conservative and for any pair $d_1,d_2\colon H_1\ra H_2$ of morphisms in $\BB$ for which there exists a contractible equalizer $f\colon F\ra P^*H_1$ of the pair $P^*d_1,P^*d_2\colon P^*H_1\ra P^*H_2$,
    there exists an equalizer $d\colon H\ra H_1$ of the pair $(d_1,d_2)$ such that $P^*d\cong f$.
\end{enumerate}
\end{theorem}

\begin{corollary}
\label{cor_finalcond}
\begin{enumerate}
\item
If the categories $\BB$ and $\CC$ are abelian then the comparison
functor $\Phi$ is fully faithful iff $\eta H\colon H\ra P_*P^*H$ is injective for
any object $H$ in $\BB$. Suppose, moreover, that the functor $P^*$
is exact. Then $\Phi$ is fully faithful iff $\Phi$ is an
equivalence and iff $P^*H\ne 0$ for any $H\ne
0\in\BB$.
\item If the categories $\BB$ and $\CC$ are triangulated
then the comparison functor $\Phi$ is fully faithful iff $\eta H\colon H\ra P_*P^*H$ is a split embedding for any object $H$ in $\BB$.
\end{enumerate}
\end{corollary}

Remark that adjoint functors between additive categories are
automatically additive.
\begin{proof}[Proof of Corollary~\ref{cor_finalcond}]
Note that in an additive category an equalizer of a pair of
morphisms $(f_1, f_2)$ is the same as a kernel of $f_1-f_2$. So for
additive categories ``to be an equalizer'' means ``to be a kernel''.

\noindent 1.
In abelian categories ``to be a kernel'' is the same as ``to be
injective''. Let us  prove the second part of the statement. If $H\ra
P_*P^*H$ is injective, then obviously $H\ne 0$ implies $P^*H\ne
0$. Vice versa, if $P^*$ does not vanish objects and $\Ker(H\ra
P_*P^*H)\ne 0$ then $\Ker(P^*H\ra P^*P_*P^*H)\ne 0$. But the morphism
$P^*H\ra P^*P_*P^*H$ is a split embedding by adjunction properties, we get a contradiction.
In order to show that $\Phi$ is an equivalence, we check conditions
of Theorem~\ref{th_beck}. Indeed, in  abelian categories equalizers
always exist and they are preserved by exact functors. The functor
$P^*$ is conservative because it is exact and does not vanish
objects.


\noindent 2.
 In triangulated categories ``to be a kernel'' is equivalent to
``to be a split embedding'', this proves the statement.
Evidently, any split embedding $f\colon H_1\to H_2$ is a kernel of the
morphism $H_2\to C(f)$.
Vice versa, suppose the morphism $f\colon H_1\to H_2$ is a kernel. Consider a distinguished triangle
$H_3\xra{h}H_1\xra{f} H_2\xra{g} H_3[1]$. Since $fh=0$ and $f$ is monomorphic, one has $h=0$. From the exact sequence
$$\Hom(H_2, H_1)\xra{f^*} \Hom(H_1, H_1)\xra{h^*=0} \Hom(H_3,H_1)$$
we deduce that there exists $f'\in \Hom(H_2, H_1)$ such that $f'f=\Id_{H_1}$,
hence $f$ splits.
\end{proof}

Recall that the category $\BB$ is called \emph{Karoubian complete} if any projector in $\BB$ splits. That is, for any object $H$ in $\BB$ and morphism $f\colon H\ra H$ such that $f^2=f$ there exist an object $H'$ called \emph{image} of $f$ and morphisms $\s\colon H'\ra H$ and $\rho \colon H\ra H'$ such that $\rho \s=\Id_{H'}$, $\s\rho=f$.

\begin{corollary}
\label{cor_suffcond} Suppose $(P^*,P_*)$ is an adjoint pair of
functors between categories $\BB$ and $\CC$. Suppose that $\BB$ is
Karoubian complete. If the natural morphism of functors $\eta \colon
\Id_{\BB}\ra P_*P^*$ splits then the comparison functor is an
equivalence.
\end{corollary}
\begin{proof}
Since $\eta\colon \Id_{\BB}\ra P_*P^*$ splits, for any object $H$ in $\BB$ there is a projector $\pi\colon
P_*P^*H\ra P_*P^*H$ whose image is $H$, and this projector depends
naturally on $H$ (i.e., we have a morphism of functors $P_*P^*\ra P_*P^*$). We need to check two conditions from Beck theorem.
First, $P^*$ is conservative: if $P^*f$ is an isomorphism for some
morphism $f$ in $\BB$ then $P_*P^*f$ is also an isomorphism, and the
same is true for $f$.
 Further, suppose that $P^*$ of some pair $(f_1,f_2)$ of morphisms in $\BB$ has a contractible equalizer.
 Then $P_*P^*$ of this pair also has a contractible equalizer. Passing to images of projectors $\pi$'s by the lemma below, we see that $(f_1,f_2)$
 also has a contractible equalizer and it is preserved by $P^*$ as pointed out below Definition~\ref{contreq}.
\end{proof}
 \begin{lemma}
Let $f_1, f_2\colon H_1\ra H_2$ be two morphisms in a Karoubian
complete category~$\BB$, suppose that they are compatible with
projectors $\pi_1\colon H_1\ra H_1, \pi_2\colon H_2\ra H_2$. Then
a contractible equalizer diagram for the pair $(f_1,f_2)$ induces a
contractible equalizer diagram on images of $\pi_1$ and~$\pi_2$.
\end{lemma}
\begin{proof}
By assumption, there exist objects $H_1'$ and $H'_2$ in $\BB$ (images of
$\pi_1$ and $\pi_2$) and morphisms $\s_i\colon H_i'\ra H_i$,
$\rho_i\colon H_i\ra H_i'$ such that $\rho_i\s_i=\Id_{H_i'}$, $\s_i\rho_i=\pi_i$
($i=1,2$). Morphisms $f_1,f_2$ induce morphisms $f_1'=\rho_2f_1\s_1$ and
$f_2'=\rho_2f_2\s_1$ from $H_1'$ to $H_2'$.
Let
$$\xymatrix{H \ar[rr]^f && H_1 \ar@/^2em/[ll]^s \ar@<1mm>[rr]^{f_1} \ar@<-1mm>[rr]_{f_2} && H_2 \ar@/^2em/[ll]^t}$$
be a contractible equalizer diagram for $(f_1,f_2)$.
Consider the morphism $\pi_1  f\colon H\ra H_1$. Since $f_1 (\pi_1
f)=\pi_2  f_1  f=\pi_2  f_2  f=f_2 ( \pi_1  f)$, by definition of an
equalizer there is a morphism $\pi\colon H\ra H$ such that $f
\pi=\pi_1  f$. This morphism $\pi$ is a projector. Indeed, $f  \pi^2=\pi_1
f \pi=\pi_1^2  f=\pi_1  f=f  \pi$, and since $f$ is mono,
$\pi^2=\pi$. By assumption, there exists an object $H'$ (an image of
$\pi$) and morphisms $\s\colon H'\ra H$, $\rho\colon H\ra H'$ such that
$\rho \s=\Id_{H'}, \s\rho=\pi$. Define $f'\colon H'\ra H_1'$ to be $\rho_1f\s$.

Let $s\colon H_1\ra H$ and $t\colon H_2\ra H_1$ be the morphisms
from the definition of a contractible equalizer. Now let $s'=\rho s\s_1$
and $t'=\rho_1t\s_2$. Easy calculations show that $f'$, $s'$ and $t'$
form a contractible equalizer diagram for the pair $(f_1',f_2')$.
\end{proof}

Here we describe a situation in which comodules over a comonad on a
triangulated category form a triangulated category.

\begin{predl}
\label{prop_cttriang} Let $(P^*,P_*)$ be a pair of adjoint exact
functors between triangulated categories $\BB$ and $\CC$, let $\TT$
be the comonad on $\CC$ associated with this pair. Suppose that $\BB$
is Karoubian complete and that the natural morphism of functors
$\eta\colon \Id_{\BB}\ra P_*P^*$ splits. Then the category $\CC_{\TT}$
is triangulated in the sense of Definition~\ref{quasitriang}.
\end{predl}
\begin{proof}
By Corollary~\ref{cor_suffcond}, the comparison functor $\Phi\colon
\BB\ra \CC_{\TT}$ is an equivalence. According to comparison
theorem, $P^*\cong Q^*\Phi$, where $Q^*\colon \CC_{\TT}\ra \CC$ is the
forgetful functor. We have to check that the diagram $F'\ra F\ra
F''\ra F'[1]$ is a triangle in $\BB$ if and only if $P^*$ of this
diagram is a triangle in $\CC$. ``Only if'' holds because $P^*$ is
exact. To prove ``if'', suppose $P^*(F'\ra F\ra F''\ra F'[1])$ is a
triangle. Then $P_*P^*(F'\ra F\ra F''\ra F'[1])$ is also a triangle,
and its direct summand $F'\ra F\ra F''\ra F'[1]$ is a triangle too (see~\cite[prop. 1.2.3]{Ne}).
\end{proof}

\section{Two ways of defining descent data}\label{section_cosimpl<>comonad}

Let $\CC_{\bul}$ be an augmented cosimplicial category satisfying assumption A1 of Section~\ref{section_cosimplicial}.
In this context two descent categories are defined.
The first one is the category $\Kern(\Augm(\CC_{\bul}))$, introduced in Section~\ref{section_cosimplicial}. To define it, one needs neither augmentation nor adjoint functors to $P_{\bul}^*$. On the contrary, the definition of the second category is fully based on categories $\CC_{-1}$ and $\CC_0$ and functors between them. This is the category of comodules over the comonad $\TT$ on $\CC_0$,
associated with the adjoint pair $(P^*,P_*)$. We recall the definition
(cf. Definition~\ref{def_comodule}).

\begin{definition}[{Comonad descent category}]
\label{def_dxt}  Objects of  $\CC_{\bul\TT}$ are pairs $(F,h)$ where
$F\in \Ob  \CC_0$ and $h\colon F\ra P^*P_*F$ is a morphism such that the composition $F\xra{h} P^*P_*F \xra{\e F} F$ is identity and the diagram
\begin{equation}
\label{equation_hh=nh}
\xymatrix{
F \ar[rr]^h \ar[d]^h  &&  P^*P_*F \ar[d]^{P^*P_*h} \\
P^*P_*F \ar[rr]^-{P^*\eta P_*F} && P^*P_*P^*P_*F }
\end{equation}
is commutative.
Morphisms from $(F_1,h_1)$ to $(F_2,h_2)$ in the category $\CC_{\bul\TT}$ are
morphisms $f\colon F_1\ra F_2$ in $\CC_0$ such that $h_2\circ
f=P^*P_*f\circ h_1$.
\end{definition}

\begin{predl}
\label{prop_twotypes} Under assumptions A1 and A2 the categories
$\Kern(\Augm(\CC_{\bul}))$ are  $\CC_{\bul\TT}$ equivalent.
\end{predl}
\begin{proof}
Objects of  $\Kern(\Augm(\CC_{\bul}))$ are pairs $(F,\theta)$, where
$F\in \Ob \CC_0$ and $\theta\colon P_1^*F\ra P_2^*F$ is a morphism
(satisfying some conditions). Likewise, objects of  $\CC_{\bul\TT}$ are
pairs $(F,h)$, where $F\in \Ob  \CC_0$ and $h\colon F\ra P^*P_*F$ is
a morphism (satisfying some other conditions). For $F\in \CC_0$ by
adjunction and base change we have
$$\Hom(P_1^*F,P_2^*F)=\Hom(F,P_{1*}P_2^*F)=\Hom(F,P^*P_*F).$$
This allows to associate an $h$ with any $\theta$ and vice versa.
Since the above isomorphisms are functorial, a map $F_1\ra F_2$ is compatible with $\theta$'s $\colon
P_1^*F_i\ra P_2^*F_i$ iff it is compatible with $h$'s: $F_i\ra
P^*P_* F_i$. All we have to do is to check that the conditions on $h$
from the definition of a comodule
\begin{itemize}
\item[(C1)] the composition
$F\xra{h} P^*P_*F \xra{\e F} F$ is identity,
\item[(C2)] the diagram~(\ref{equation_hh=nh})  commutes,
\end{itemize}
are equivalent to the following conditions:
\begin{itemize} \item[(C1')] $\theta$ is an isomorphism,
\item[(C2')] the cocycle condition on $\theta$ holds: morphisms
$P_{13}^*\theta$ and $P_{23}^*\theta\circ P_{12}^*\theta$ from
$P_{13}^*P_1^*F$ to  $P_{23}^*P_2^*F$ are equal.
\end{itemize}

First we show that (C2) is equivalent to (C2'). One has
\begin{align*}
\Hom(P_{13}^*P_1^*F,P_{23}^*P_2^*F)&=
\Hom(F,P_{1*}P_{13*}P_{23}^*P_2^*F)&&\text{by adjunction}\\
&=\Hom(F,P_{1*}P_{12*}P_{23}^*P_2^*F)&&\text{by cosimplicial relations}\\
&=\Hom(F,P_{1*}P_2^*P_{1*}P_2^*F)&&\text{by base change}\\ &=\Hom(F,P^*P_*P^*P_*F)&&\text{by base change}.
\end{align*}
Under these identification
the morphism $P_{13}^*\theta\in
\Hom(P_{12}^*P_1^*F,P_{23}^*P_2^*F)$ corresponds to
the morphism
$P^*\eta P_* F\circ h\in\Hom(F,P^*P_*P^*P_*F)$ and
the morphism
$P_{23}^*\theta\circ P_{12}^*\theta\in
\Hom(P_{12}^*P_1^*F,P_{23}^*P_2^*F)$ corresponds to
the morphism $P^*P_*h\circ
h\in\Hom(F,P^*P_*P^*P_*F)$.
Thus (C2) is equivalent to (C2').

Now we prove that (C1')+(C2') imply (C1). First note that the
morphism $f\colon F\xra{h}P^*P_*F\xra{\e F} F$ equals the pull-back $D^*\theta\colon
F=D^*P_1^*F\ra D^*P_2^*F=F$.  Since $\theta$ is
an isomorphism, $f$ is also an isomorphism. Further, the cocycle
condition for $\theta$ implies that $f^2=f$.  Thus $f=\Id_F$.

It remains to check that (C1)+(C2) imply (C1'). Recall that $\theta$
is obtained from $h$ by adjunction in the following way:
$$\theta\colon P_1^*F\xra{P_1^*h} P_1^*P^*P_*F\xra{\sim} P_2^*P^*P_*F\xra{P_2^*\e} P_2^*F$$
One can check that the map
$$\theta'\colon P_2^*F\xra{P_2^*h} P_2^*P^*P_*F \xra{\sim} P_1^*P^*P_*F\xra{P_1^*\e} P_1^*F$$
is inverse to $\theta$, therefore $\theta$ is an isomorphism.
\end{proof}

Proposition~\ref{prop_twotypes} says that the category $\CC_{\bul\TT}$ does not depend on augmentation and depends only on cosimplicial part
$[\CC_0,\CC_1,\ldots,P_f^*]$. In fact, the comonad $\TT$ also does not depend on augmentation.

\begin{corollary}
\label{cor_simpl>comonad}
For any cosimplicial category $\CC_{\bul}=[\CC_{0}, \CC_1, \ldots,
P_{\bul}^*]$ satisfying assumptions A1 and A2 there is a well-defined comonad
$\TT$ on the category $\CC_0$. It is isomorphic to the comonad constructed in
Definition~\ref{def_dxt} for any extension of $\CC_{\bul}$  to an augmented  cosimplicial category satisfying A1 and A2.
\end{corollary}
\begin{proof}
Let $\~{\CC}_{\bul}$ be an extension of $\CC_{\bul}$ to an augmented cosimplicial category, satisfying A1 and A2. (For example, one can take the category, constructed in Proposition~\ref{prop_augmcosimplcat}, as $\~{\CC}_{\bul}$).
Apply Definition~\ref{def_dxt} to $\~{\CC_{\bul}}$ to construct a comonad.
We see that the functor
$T=P^*P_*=P_{2*}P_1^*$ does not depend on augmentation.
It is not hard to check that the natural transformations of functors
$T=P^*P_*\ra \Id$ and $T=P^*P_*\ra P^*P_*P^*P_*=TT$ have the form
$$P_{2*}P_1^*\xra{\eta} P_{2*}D_*D^*P_1^*\xra{\sim}\Id\circ \Id=\Id$$
and
$$P_{2*}P_1^*\xra{\eta} P_{2*}P_{13*}P_{13}^*P_1^* \xra{\sim}
P_{2*}P_{23*}P_{12}^*P_1^*\xra{\sim}
 P_{2*}P_1^*P_{2*}P_1^*,$$
and hence do not depend on augmentation.
\end{proof}

\section{Restriction to subcategories}
\label{s5}

In this section we collect some facts concerning
subcategories in descent categories.

Suppose there is a cosimplicial subcategory
$$\CC'_{\bullet}=[ \CC'_{0}, \CC'_1, \ldots,
P_{\bullet}^*]\qquad\text{in}\qquad [\CC_{0}, \CC_1, \ldots,
P_{\bullet}^*].$$ This means that there is a collection of subcategories $\CC'_i$ in each
$\CC_i$ compatible with the pullbacks $P_{\bullet}^*$, i.e.
$P_{\bullet}^*\CC'_m\subset\CC'_n$ (and not necessarily compatible
with the pushforwards $P_{\bullet *}$). In this context one can consider the classical descent category
$\Kern(\CC'_{\bullet})$, it is a subcategory of $\Kern(\CC_{\bullet})$.

\begin{remark}
\label{remark_kern'}
Note that the category $\Kern(\CC'_{\bullet})$ in fact does not
depend on $\CC'_1, \CC'_2, \ldots$. It can be defined  as soon as a
subcategory $\CC'_0\subset \CC_0$ is given: one can take
$\CC'_k=\CC_k$ for $k>0$.
\end{remark}

Suppose $\TT$ is a comonad on the category $\CC$ and
$\CC'\subset\CC$ is a subcategory. Define the category $\CC'_{\TT}$
as follows.
\begin{definition}
\label{def_dxt'} Objects of $\CC'_{\TT}$ are  pairs $(F,h)$ in $\Ob
\CC_{\TT}$ such that $F\in \Ob  \CC'$, and morphisms in $\CC'_{\TT}$ are
morphisms in $\CC_{\TT}$ that lie in $\CC'$:
$$\Hom_{\CC'_{\TT}}((F_1,h_1),(F_2,h_2))=
\Hom_{\CC_{\TT}}((F_1,h_1),(F_2,h_2))\cap \Hom_{\CC'}(F_1,F_2).$$
Evidently, $\CC'_{\TT}$ is a subcategory of $\CC_{\TT}$, it is full
if $\CC'$ is a full subcategory of $\CC$.
\end{definition}
The motivation for this definition is the following: in many important
cases the functor $T\colon \CC\ra \CC$ is rather ``big'' and does not
preserve ``small'' subcategories in $\CC$. Therefore,
Definition~\ref{def_comodule} does not allow to construct a category of
``small'' objects equipped with descent data. The typical example here is when
$p\colon X\ra S$ is a non-proper morphism of schemes, $\CC=\qcoh(X),
\CC'=\coh(X)$ and $T=p^*p_*$.

A straightforward analog of Proposition~\ref{prop_adab} holds:
\begin{predl}
\label{prop_adabtr} Suppose $\TT=(T,\e,\delta)$ is a comonad on a
category $\CC$.

If $\CC'\subset\CC$ are additive categories and $T$ is an additive
functor then $\CC'_{\TT}$ is also additive.

If $\CC'\subset\CC$ are abelian and $T$ is left exact then
$\CC'_{\TT}$ is abelian too.

If $\CC'\subset\CC$ are triangulated categories, the functor $T$ is
exact and $\CC_{\TT}$ is triangulated in the sense of
Definition~\ref{quasitriang}, then $\CC'_{\TT}\subset \CC_{\TT}$ is
a triangulated subcategory.
\end{predl}

Let  $\CC_{\bul}$ be an augmented cosimplicial category satisfying assumptions A1 and A2 from Section~\ref{section_cosimplicial}, let $\CC'_{\bul}\subset \CC_{\bul}$ be a cosimplicial subcategory (possibly, without augmentation). In this context two descent categories related with the subcategory  $\CC'_{\bul}$ are defined -- the category $\Kern(\CC'_{\bul})$ and the category $\CC'_{\bul\TT}=\CC'_{0\TT}$ defined above.
An easy corollary of
Proposition~\ref{prop_twotypes} holds:
\begin{corollary}
\label{cor_twotypes} In the above notation the categories
$\Kern(\CC'_{\bullet})$ and $\CC'_{\bul\TT}$ are equivalent.
\end{corollary}

Let  $(P^*,P_*)$ be an adjoint pair of functors between $\BB$ and
$\CC$ and $\TT$ be the associated comonad on $\CC$. If $P^*\colon \BB\to \CC$ takes a subcategory $\BB'\subset \BB$ into $\CC'\subset \CC$, one can
consider the restriction of the comparison functor
$$\Phi\mid_{\BB'}\colon \BB'\ra\CC'_{\TT}.$$
In particular, one can take $\BB'$ to be the preimage $(P^*)^{-1}(\CC')$:
it is the subcategory in $\BB$ whose objects/morphisms are exactly the
objects/morphisms of $\BB$ that are sent by $P^*$ into objects/morphisms of $\CC'$.
\begin{lemma}
\label{lemma_restrictcomparison}
 If the comparison functor $\Phi\colon
\BB\ra \CC_{\TT}$ is an equivalence and $\CC'\subset\CC$ is a
subcategory, then the restriction of $\Phi$ on
$\BB'=(P^*)^{-1}(\CC')$ is an equivalence $ \BB'\ra \CC'_{\TT}$.
\end{lemma}
\begin{proof}
Obvious.
\end{proof}

\section{Coherent sheaves on schemes and stacks and their derived categories}
\label{s6}
When coherent and quasi-coherent sheaves on a scheme are considered, the scheme is usually supposed to be quasi-compact and quasi-separated. Recall (see.~\cite[1.1,1.2]{EGA}) that a scheme is quasi-compact if it can be covered by a finite number of open affine subschemes, a scheme is quasi-separated if the intersection of any two its open affine subschemes can be covered by a finite number of open affine subschemes. For instance, any Noetherian (in particular, any quasi-projective)
scheme is quasi-compact and quasi-separated. From now on all schemes
are supposed to be quasi-compact and quasi-separated.

In this section we give a review of coherent and quasi-coherent sheaves on stacks and of derived categories of sheaves on stacks. More detailed exposition of the subject can be found in papers by Laumon and
Moret-Bailly~\cite{La}, Laszlo and Olsson~\cite{LaOl} and Arinkin
and Bezrukavnikov~\cite{AB}. All stacks in this paper are supposed to be algebraic stacks of finite type over a field. In particular, any stack $X$ is assumed to be Noetherian,  quasi-compact and quasi-separated, it can be covered by a scheme of finite type over a field. By sheaves on $X$ we understand sheaves of $\O_X$-modules in smooth topology (see~\cite[6.1]{La}).

For all the paper we accept the following
$$
\parbox{0.9\textwidth}{{\bf Assumption.}
Any stack is supposed to be either an algebraic stack of finite type over an arbitrary field~$\k$ or a quasi-compact quasi-separated scheme.} \leqno{(\dagger)}
$$

Suppose $X$ is a stack.  The category $\O_X\Mod$ of sheaves of
$\O_X$-modules on $X$ is a Grothendieck category, it has enough injective objects (see~\cite[chapter 18]{KSh}).  The unbounded derived category of an abelian category
$\O_X\Mod$ is denoted by $\D(\O_X\Mod)$, it has arbitrary direct sums. Categories of quasi-coherent and coherent sheaves of $\O_X$-modules are denoted by
$\qcoh(X)$ and $\coh(X)$ respectively. Consider $\D_{\qcoh}(X)$, a full subcategory in $\D(\O_X\Mod)$, formed by complexes with quasi-coherent cohomology sheaves. This category is also closed under taking direct sums.
Denote by  $\D_{\coh}^b(X)\subset
\D_{\qcoh}(X)$ a full subcategory, formed by complexes, whose cohomology sheaves are coherent and almost all equal to zero.
Other versions of derived categories are defined similarly.
By~\cite[claim 2.5]{AB}, one can use the smooth-\'etale topology on $X$
instead of the smooth, and in the case when $X$ is a scheme -- the Zariski topology, this will not affect the (quasi)coherent categories
$\qcoh(X)$, $\coh(X)$, $\D_{\qcoh}(X)$, etc.

Suppose that stacks $X$ and $Y$ satisfy Assumption~$(\dagger)$, and
$f\colon X\ra Y$ is a morphism. This gives a  morphism of ringed smooth sites $(X,\O_X)\ra (Y,\O_Y)$ and functors
$f^*\colon \O_Y\Mod\ra\O_X\Mod$ and $f_*\colon
\O_X\Mod\ra\O_Y\Mod$, the functor $f^*$ is left adjoint to $f_*$. The functor $f^*$
preserves quasi-coherence  (\cite[6.8]{La}) and coherence. Further, if $f$ is quasi-compact and quasi-separated,
then $f_*$ preserves quasi-coherent sheaves, while this  is not true for coherent sheaves in general.

Following Spaltenstein~\cite{Sp}, one can define derived functors of direct and inverse image on unbounded derived categories.
By~\cite[prop. 2.1.4]{LaOl}, any complex $F$ in
$\D_{\qcoh}(X)$ has a K-injective resolution. Applying the functor $f_*$ termwise to this resolution, we obtain an object of the category $\D(\O_Y\Mod)$. This defines a functor
$Rf_*\colon \D_{\qcoh}(X)\ra\D(\O_Y\Mod)$. If the morphism $f$
is quasi-compact and quasi-separated and  $F\in \D^+_{\qcoh}(X)$ then
\cite[6.8]{La} implies that $Rf_*F\in\D^+_{\qcoh}(X)$. If, in addition, the morphism $f$ is representable then $Rf_*$ has a finite cohomological dimension.
Then it follows from~\cite[lemma 2.1.10]{LaOl} that the functor
$Rf_*$ sends $\D_{\qcoh}(X)$ into $\D_{\qcoh}(Y)$.

The derived pull-back functor $Lf^*\colon \D(\O_Y\Mod)\ra\D(\O_X\Mod)$ on unbounded derived categories is defined using K-flat resolutions, see~\cite[section 18.6]{KSh}. It sends $\D_{\qcoh}(Y)$ into $\D_{\qcoh}(X)$ and $\D^-_{\coh}(Y)$ into $\D^-_{\coh}(X)$ (see~\cite[6.8, 8.7]{La}).  Derived push-forward and pull-back functors on unbounded derived categories satisfy all expected properties (see~\cite[section 18.6]{KSh}):
they are adjoint to each other,  $R(fg)_*\cong Rf_*Rg_*$ and analogously
for pull-back functors, the projection formula and the base change formula hold as well. For representable morphisms these functors commute with arbitrary direct sums.

If a stack $X$ is quasi-compact and semi-separated (i.e. the diagonal morphism $X\ra X\times X$ is affine) then the category $\D^+_{\qcoh}(X)$
is equivalent to $\D^+(\qcoh(X))$, the bounded below derived category of an abelian category $\qcoh(X)$. Likewise,
$\D_{\coh}^b(X)$ is equivalent to $\D^b(\coh(X))$, see~\cite[claim 2.7, cor. 2.11]{AB}.
If, in addition,  $X$ is a scheme then there is also an equivalence of the unbounded derived categories: $\D_{\qcoh}(X)\cong \D(\qcoh(X))$,
see~\cite[cor. 5.5]{BN}. In this paper we will mostly deal with the ``big'' category $\D_{\qcoh}(X)$ and, in the Noetherian case, with
the ``small'' category $\D_{\coh}^b(X)$. We introduce the following notation:
$$\D(X)=\D_{\qcoh}(X)
.$$

A \emph{perfect complex} on a stack or a scheme $X$ is a complex,
locally quasi-isomorphic to a bounded complex of vector bundles,
i.e. such $F\in \D(\O_X\Mod)$ that for a certain covering $f\colon U\ra X$
by a scheme the pull-back $f^*F$
is quasi-isomorphic to a bounded complex of locally free sheaves of finite rank.
A full subcategory in  $\D(\O_X\Mod)$ formed by perfect complexes is denoted by
$\D^{\perf}(X)$. Obviously,
$\D^{\perf}(X)\subset \D_{\qcoh}(X)$, and for a Noetherian stack
$\D^{\perf}(X)\subset \D^b_{\coh}(X)$. Note that the pull-back functor
$Lf^*\colon \D(\O_Y\Mod)\ra\D(\O_X\Mod)$ sends perfect complexes into perfect complexes.

If $X$ is a scheme, by~\cite[th. 3.1.1]{BvdB}, the category of perfect complexes on $X$ is exactly the category of compact objects in $\D_{\qcoh}(X)$. It generates the category $\D_{\qcoh}(X)$. For a smooth scheme over a field
the category of perfect complexes coincides with the bounded derived category of coherent sheaves.

Recall that a morphism of stacks is \emph{strictly flat},
if it is flat and surjective. Equivalently, one can say that a flat morphism $f\colon X\ra S$ is strictly flat if and only if for any sheaf
$H$ on $S$ such that $f^*H=0$ one has $H=0$. Also it is true that $f$ is strictly flat if and only if $f$ is flat and the pull-back functor $f^*$ is conservative.

\section{Derived descent theory for stacks}
\label{s7} In this section we apply the results of Sections~\ref{section_comonad} and~\ref{section_cosimpl<>comonad} to studying cohomological descent for derived categories of sheaves on schemes and stacks. We
work in the category of stacks in order to treat the following two cases
simultaneously: the descent for morphisms of schemes and the descent for
equivariant derived categories.

Recall that we work under Assumption~$(\dagger)$: all stacks are either algebraic stacks of finite type over a field or quasi-compact quasi-separated schemes.
Any time we deal with coherent sheaves on a stack, the stack is supposed to be Noetherian.

Let $X$ and $S$ be stacks and $p\colon X\ra S$ be a flat representable morphism. Consider the following augmented simplicial stack
$$(X\ra S)_{\bul}=[S,X,X\times_SX,X\times_SX\times_SX,\ldots,p_{\bul}].$$
All fibred products $X\times_SX,X\times_SX\times_SX,\ldots$ satisfy Assumption~$(\dagger)$.
The abelian categories of quasi-coherent sheaves on $S$, $X$,
$X\times_SX\ldots$ and the pull-back functors between them form an augmented cosimplicial category
\begin{equation}
\label{equation_qcohX}
[\qcoh(S),\qcoh(X),\qcoh(X\times_SX),\qcoh(X\times_SX\times_SX),\ldots,p^*_{\bul}].
\end{equation}
This category satisfies assumptions A1 and A2 of Section~\ref{section_cosimplicial}:
the functors $p_{\bul}^*$ have right adjoint functors~$p_{\bul *}$ and the flat base change formula holds.
Consider the augmented cosimplicial subcategory of~(\ref{equation_qcohX}), formed by categories of coherent sheaves:
\begin{equation}
\label{equation_cohX}
[\coh(S),\coh(X),\coh(X\times_SX),\coh(X\times_SX\times_SX),\ldots,p^*_{\bul}].
\end{equation}
Note that this category does not satisfy assumption A1 of Section~\ref{section_cosimplicial}:
the push-forward functors do not preserve coherent sheaves. Denote by
$$\qcoh(X)/p=\Kern([\qcoh(X),\qcoh(X\times_SX),\ldots])$$ and
$$\coh(X)/p=\Kern([\coh(X),\coh(X\times_SX),\ldots])$$
the descent categories associated with~(\ref{equation_qcohX})
and~(\ref{equation_cohX}) as in Definition~\ref{def_dxbar}.
\begin{theorem}
\label{th_abdescentforstacks}
Let $X$ and $S$ be stacks satisfying Assumption~$(\dagger)$, and $p\colon X\ra S$ be a flat representable morphism. Then the category $\qcoh(X)/p$ is equivalent to the category $\qcoh(S)$ of quasi-coherent sheaves on $S$ if and only if the morphism
$p$ is strictly flat. For Noetherian stacks $S$ and $X$ these conditions hold if and only if the categories $\coh(X)/p$ and $\coh(S)$ are equivalent.
\end{theorem}
\begin{remark} In the case of schemes this result is well-known
(see, for example,~\cite[prop. 2.22]{Miln}). We give the proof just to demonstrate how Corollary~\ref{cor_finalcond} works here.
\end{remark}
\begin{proof}
Denote by $\qcoh(X)_{\TT_p}$ the descent category, associated with the comonad
$\TT_p=(p^*p_*,\e,\delta)$ on $\qcoh(X)$ (see Definition~\ref{def_dxt}).
The category~(\ref{equation_qcohX}) satisfies assumptions~A1 and~A2
of Section~\ref{section_cosimplicial}, therefore by Proposition~\ref{prop_twotypes}
the categories ${\qcoh(X)}/p$ and
$\qcoh(X)_{\TT_p}$ are equivalent.
Now we apply Corollary~\ref{cor_finalcond}.1. The functor $p^*$ is exact,
hence the comparison functor $\Phi\colon\qcoh(S)\ra
\qcoh(X)_{\TT_p}$ is an equivalence iff
for any
$H\in\qcoh(S)$ such that $p^*H=0$ one has $H=0$. The latter condition is equivalent to $p$ being strictly flat.

Now assuming $S$ and $X$ are Noetherian we want to show that the comparison functor is an equivalence on coherent categories  if and only if it is an equivalence on quasi-coherent categories. Denote by $\coh(X)_{\TT_p}$ the subcategory in $\qcoh(X)_{\TT_p}$,
corresponding to the subcategory $\coh(X)\subset \qcoh(X)$ (see
Definition~\ref{def_dxt'}). By Corollary~\ref{cor_twotypes},
the categories $\coh(X)/p$ and $\coh(X)_{\TT_p}$ are equivalent.

Suppose that the functor $\Phi\colon \qcoh(S)\ra \qcoh(X)_{\TT_p}$ is an equivalence. Let us check that the restriction of $\Phi$ is an equivalence between strictly full subcategories $\coh(S)\subset
\qcoh(S)$ and $\coh(X)_{\TT_p}\subset \qcoh(X)_{\TT_p}$. Lemma~\ref{lemma_restrictcomparison} claims that the restriction of~$\Phi$
is an equivalence between $(p^*)^{-1}(\coh(X))$ and $\coh(X)_{\TT_p}$. Therefore
we need to check that for any $H\in \qcoh(S)$ the sheaf $p^*H$ is coherent if and only if the sheaf $H$ is coherent. This is Lemma~\ref{lemma_p*coh}.

Conversely, suppose that the comparison functor $\Phi|_{\coh(S)}\colon
\coh(S)\ra {\coh(X)}_{\TT_p}$ is an equivalence. Let us show that the functor
$\Phi\colon \qcoh(S)\ra {\qcoh(X)}_{\TT_p}$ is also an equivalence.
We claim that for any sheaf $H\in \coh(S)$ the morphism $H\xra{\eta H} p_*p^*H$ is injective. Indeed, consider the kernel
$K=\ker(H\ra p_*p^*H)$. The morphism
$p^*\eta H\colon p^*H\ra p^*p_*p^*H$ is a split embedding, its inverse map is  a canonical adjunction morphism $\e p^*H\colon p^*p_*p^*H\ra p^*H$. Since $p^*$ is left exact, it follows that $p^*K=0$, hence $\Phi(K)=0$. But the sheaf $K$ on $S$ is coherent, so $K=0$.
Then, any quasi-coherent sheaf is a filtered colimit of its coherent subsheaves. Since filtered colimit is left exact, the map $H\to p_*p^*H$ is injective for all quasi-coherent sheaves $H$ on $S$ as well.
Applying Corollary~\ref{cor_finalcond}.1
we conclude that $\Phi\colon \qcoh(S)\ra
\qcoh(X)_{\TT_p}$ is an equivalence.
\end{proof}

An analogue of Theorem~\ref{th_abdescentforstacks} for derived categories is more
interesting.

Consider the following augmented cosimplicial category
\begin{equation} \label{equation_DqcohX}
[\D(S),\D(X),\D(X\times_SX),\D(X\times_SX\times_SX),\ldots,Lp^*_{\bul}],
\end{equation}
formed by the unbounded derived categories of quasi-coherent sheaves on stacks $S$, $X$, $X\times_SX,\ldots$ Denote by
$$\D(X)/p=\Kern([\D(X),\D(X\times_SX),\D(X\times_SX\times_SX),\ldots,Lp_{\bul}^*])$$
the classical descent category (see Definition~\ref{def_dxbar}).
The cosimplicial category~(\ref{equation_DqcohX}) satisfies assumptions A1 and
 A2 of Section~\ref{section_cosimplicial}: the functors
$Lp_{\bul}^*$ have right adjoint functors $Rp_{\bul *}$, the flat base change formula holds. Consider also the classical descent categories
$\D^b_{\coh}(X)/p$ и $\D^{\perf}(X)/p$ (see
Definition~\ref{def_dxbar} and Remark~\ref{remark_kern'}),
associated with~(\ref{equation_DqcohX}) and subcategories $\D^b_{\coh}(X)$
and $\D^{\perf}(X)$ in $\D(X)$.
\begin{theorem}
\label{th_descentforstacks}
Let $X$ and $S$ be stacks satisfying Assumption~$(\dagger)$ and  $p\colon X\ra S$ be a flat representable morphism.
Then the comparison functor $\D(S)\ra \D(X)/p$
is an equivalence if and only if the natural map $\O_S\ra Rp_*\O_X$ in the category
$\D(S)$ is an embedding of a direct summand.
Under this condition the comparison functors
$\D^{\perf}(S)\ra \D^{\perf}(X)/p$ and (for Noetherian $X$ and $S$)
$\D^b_{\coh}(S)\ra{\D^b_{\coh}(X)}/p$ are equivalences.
\end{theorem}
\begin{remark} Note that the splitting of the map
$\O_S\ra Rp_*\O_X$ is not necessary for the functor
$\D^b_{\coh}(S)\ra\D^b_{\coh}(X)/p$ to be an equivalence, see Example~\ref{example_a1} below.
\end{remark}
\begin{proof}
Denote by $\D(X)_{\TT_p}$ the category of comodules over the comonad
$\TT_p=(p^*Rp_*,\e,\delta)$  on the category $\D(X)$ (see Definition~\ref{def_dxt}). By Proposition~\ref{prop_twotypes}, the descent categories
$\D(X)/p$ and $\D(X)_{\TT_p}$ are equivalent. We will use results from comonad theory.

Suppose that the comparison functor $\Phi\colon \D(S)\ra\D(X)_{\TT_p}$ is an equivalence, then it is fully faithful and due to Corollary~\ref{cor_finalcond}.2
the map $\O_S\ra
Rp_*p^*\O_S=Rp_*\O_X$ is a split embedding. The converse follows from the Corollary~\ref{cor_suffcond}. Indeed, the category $\D(S)$ is Karoubian complete. By the projection formula, for any  $H\in\D(S)$ the natural morphism
$H\ra Rp_*p^*H\cong H\otimes^LRp_*\O_X$ has the form
$H\otimes^L(\O_S\ra Rp_*\O_X)$. Consequently, the morphism of functors
$\Id\ra Rp_*p^*$ is split if the map
$\O_S\ra Rp_*\O_X$ is split.

Now suppose that the functor
$$\Phi\colon\D(S)\ra\D(X)_{\TT_p}$$ is an equivalence.
Denote by $\D^b_{\coh}(X)_{\TT_p}$ and $\D^{\perf}(X)_{\TT_p}$
the comonad descent categories associated with the subcategories
$\D^b_{\coh}(X)\subset\D(X)$ and $\D^{\perf}(X)\subset\D(X)$. As before,
Corollary~\ref{cor_twotypes} implies that
$\D^b_{\coh}(X)/p\cong\D^b_{\coh}(X)_{\TT_p}$ and $\D^{\perf}(X)/p\cong\D^{\perf}(X)_{\TT_p}$.
It remains to prove that the restriction of
$\Phi$ is an equivalence between strict full subcategories
$\D^b_{\coh}(S)\subset \D(S)$
and $\D^b_{\coh}(X)_{\TT_p}\subset \D(X)_{\TT_p}$. According to
Lemma~\ref{lemma_restrictcomparison}, the restriction of $\Phi$ is an equivalence between $(p^*)^{-1}(\D^b_{\coh}(X))$ and  $\D^b_{\coh}(X)_{\TT_p}$.
So we need to check that for any $H\in \D(S)$ the complex
$p^*H$ is in $\D^b_{\coh}(X)$ if and only if $H$ is in $\D^b_{\coh}(S)$. This follows from Lemma~\ref{lemma_p*coh} below and the exactness of the functor
$p^*$.

Arguing in the same manner in the case of perfect complexes we need to show that
$(p^*)^{-1}(\D^{\perf}(X))=\D^{\perf}(S)$.  Evidently, the pull-back of a perfect complex is a perfect complex. Let us check that the converse is true.
First, we reduce to the case of schemes.

Choose a covering $f\colon U\ra S$ of a stack $S$ by a scheme. The morphism $p$ is representable, hence the map $f'\colon U'=U\times_SX\ra X$ is also a covering by a scheme. By definition, an object $H\in\D(S)$ is a perfect complex on $S$ if and only if $f^*H$ is a perfect complex on $U$, the same is true for
$f'\colon U'\ra X$.
Perfect complexes on a scheme $U$ are precisely compact objects in the category
$\D(U)$ (see~\cite[3.1.1]{BvdB}), the same is true for~$U'$.
Thus we need to show that for a morphism of schemes
$p'\colon U'\ra U$ an object $H\in\D(U)$ is compact if the object
$p'^*H$ is compact in $\D(U')$.

Suppose $H\in\D(U)$ and $p'^*H$ is compact in $\D(U')$. Let $(F_{\a})$ be an arbitrary family of objects in $\D(U)$.
Consider the following  commutative diagram
\begin{equation}
\label{equation_Hcomp}
\vcenter{\xymatrix{ {\bigoplus \Hom(H,F_{\a})} \ar[r]
\ar[d] & {\bigoplus
\Hom(H,Rp'_*p'^*F_{\a})}\ar[d] \\
\Hom\left(H,\bigoplus F_{\a}\right)\ar[r] & \Hom\left(F,Rp'_*p'^*\left(\bigoplus F_{\a}\right)\right).
}} \end{equation}
Note that the sheaf $\O_U$ is a direct summand in
$Rp'_*\O_{U'}$ (see Proposition~\ref{predl_SCDTproperties}), therefore
the functor $\Id_{\D(U)}$ is a direct summand in $Rp'_*p'^*$. We conclude that the left column in~(\ref{equation_Hcomp}) is a direct summand of the right column.
Let us show that the morphism in the right column is an isomorphism.
We have:
\begin{align*}
\bigoplus\Hom(H,Rp'_*p'^*F_{\a})&=\bigoplus\Hom(p'^*H,p'^*F_{\a})&&\text{by adjunction}\\
&=\Hom\left(p'^*H,\bigoplus
p'^*F_{\a}\right)&&\text{because $p'^*H$ is compact}\\
&=\Hom\left(p'^*H,p'^*\left(\bigoplus F_{\a}\right)\right)&&\text{because $p'^*$ commutes with $\oplus$}\\
&=\Hom\left(H,Rp'_*p'^*\left(\bigoplus F_{\a}\right)\right)&&\text{by adjunction.}
\end{align*}
So the left column in~(\ref{equation_Hcomp}) is also an isomorphism,
and $H$ is a compact object of $\D(U)$.
\end{proof}

\begin{lemma}
\label{lemma_p*coh} Let $X,S$ and $p$ be as in Theorem~\ref{th_descentforstacks}.
Suppose that the functor $\Phi\colon \qcoh(S)\ra
\qcoh(X)_{\TT_p}$ is an equivalence. If $H$ is a quasi-coherent sheaf on  $S$
and the sheaf $p^*H$ is coherent then $H$ is also coherent.
\end{lemma}
\begin{proof}
By~\cite[prop. 15.4]{La}, the sheaf $H$ is a union of its coherent subsheaves.
If~$H$ is not coherent then one can choose a strictly monotonous sequence  $H_1\ra H_2\ra H_3\ra\ldots$ of coherent subsheaves in~$H$. Since $\Phi$ is an equivalence, the sequence
$p^*H_i$ of subsheaves in $p^*H$ is also strictly monotonous. But the stack $X$ is Noetherian and the sheaf
$p^*H$ is coherent, that gives a contradiction.
\end{proof}

\section{SCDT morphisms}
\label{s8}

Let $p\colon X\ra S$ be a flat morphism.
Theorem~\ref{th_descentforstacks} shows that the splitting of the map $\O_S\to Rp_*\O_X$
is a criterion for the derived category of $S$ to be equivalent to the descent category associated with the morphism $p$.
Therefore the above property is of some interest.

 \begin{definition}
\label{def_SCDT}We say that the morphism of schemes or stacks
$p\colon X\ra S$ is \emph{SCDT}, if $p$ is flat and the natural morphism
$\O_S\ra
Rp_*\O_X$ is an embedding of a direct summand in the category of sheaves of $\O_S$-modules.  SCDT stands for``strictly cohomological descent type''.
\end{definition}

In this section we collect some facts about SCDT morphisms and some sufficient conditions for a morphism to be SCDT. Recall
that a functor $\Phi$ is said to be \emph{faithful} if for any pair of
objects $A,B$ the induced mapping
$$\Hom(A,B)\ra\Hom(\Phi(A),\Phi(B))$$ is injective.

\begin{lemma}
Let $X$ and $S$ be stacks satisfying Assumption~$(\dagger)$, and $p\colon X\ra S$ be
a flat morphism. Then $p$ is SCDT iff the functor $p^*\colon \D(S)\ra\D(X)$ is faithful.
\end{lemma}

\begin{proof}
Suppose $p$ is SCDT. By the projection formula, the functor $\Id\colon \D(S)\ra\D(S)$ is a direct summand of the functor
$Rp_*p^*$. Therefore, $Rp_*p^*$ does not vanish on morphisms, so neither does~$p^*$.

Suppose $p^*$ is faithful. Consider the distinguished triangle
$\O_S\ra Rp_*\O_X\ra K\xra{f} \O_S[1]$. Let us check that $f=0$.
Applying $p^*$ to this triangle we obtain $p^*\O_S\ra p^*Rp_*p^*\O_S\ra
p^*K\xra{p^*f} p^*\O_S[1]$. Adjunction properties of $p^*$ and $Rp_*$ imply that
the latter triangle splits, so $p^*f=0$. But since $p^*$ is faithful, we have $f=0$.
\end{proof}

\begin{predl}
\label{predl_SCDTproperties}
Let $S,S',X,Y$ be stacks satisfying Assumption~$(\dagger)$.
\begin{enumerate}
\item If morphisms $q\colon Y\ra X$ and $p\colon X\ra S$ are SCDT
then $p\circ q$ is also SCDT. If $p$ is flat and $p\circ q$ is
SCDT then $p$ is also SCDT. If $q$ is SCDT and $p\circ q$ is flat then $p$ is also flat.
\item SCDT property is stable under base change: if a morphism
$p\colon X\ra S$ is SCDT, and $s\colon
S'\ra S$ is a base change, then the morphism $p'\colon X'=X\times_SS'\ra
S'$ is also SCDT.
\item If, in addition, the base change morphism is SCDT then $p$ and $p'$ are simultaneously SCDT (or not SCDT).
\item Let $\k\subset K$ be a field extension. The morphism
$p\colon X\ra S$ of stacks over $\k$ is SCDT if and only if the morphism $p'\colon X\times_{\k}K\ra
S\times_{\k}K$ is SCDT.
\end{enumerate}
\end{predl}
\begin{proof} The first statement in 1 is trivial. Let us prove the second. The composition of natural morphisms $$\O_S\xra{\eta_p} Rp_*\O_X\xra{Rp_*\eta_q} Rp_*Rq_*\O_Y$$ splits, let $\s\colon Rp_*Rq_*\O_Y\ra \O_S$ be a splitting morphism. Then $\s\circ Rq_*\eta_q$ is a splitting morphism for $\eta_p$. Let us prove the third statement. For this we need to check that the functor $p^*\colon \qcoh(S)\ra\qcoh(X)$ is exact. Suppose $H_1\ra H_2\ra H_3$ is an exact sequence of quasi-coherent sheaves on $S$, and $F$  is a middle term cohomology of the sequence
$p^*H_1\ra p^*H_2\ra p^*H_3$. Since the functors $q^*$ and $(p\circ q)^*$ are exact, we have $q^*F=0$. Since $q$ is SCDT, we conclude that $F=0$.

Part 2 follows from the flat base change formula (see~\cite[2.4]{Ku}): $Rp'_*\O_{X'}=s^*Rp_*\O_X$,
part 3~follows from 1 and 2, part 4~follows from 3. \end{proof}

\begin{predl}
\label{predl_finiteflat}
A finite flat morphism $p\colon X\ra S$ of quasi-compact quasi-separated schemes over a field of zero characteristic is SCDT.
\end{predl}
\begin{proof}
Since $p$ is affine, we have $Rp_*\O_X=p_*\O_X$ (higher direct images vanish).
Denote the quotient $p_*\O_X/\O_S$ by $C$. Let us demonstrate that the extension
$0\ra \O_S\ra p_*\O_X\ra C\ra 0$ splits.
Tensoring it by $p_*\O_X$, one obtains a split extension. Indeed, the mapping
$p_*\O_X\ra p_*\O_X\otimes p_*\O_X$ splits, an inverse mapping is given by multiplication in the sheaf $p_*\O_X$ of $\O_S$-algebras.
The morphism $p$ is finite and flat, hence the sheaf $E=p_*\O_X$ is a vector bundle on~$S$. Tensoring by $E$ induces a mapping
$$\s\colon \Ext^1(C,\O_S)\ra \Ext^1(C\otimes E,\O_S\otimes E)=\Ext^1(C,\EEnd(E)).$$
One can check that $\s$ is a monomorphism: the left inverse to $\s$ is given by
the trace $\EEnd(E)\ra \O_S$:
$$\Ext^1(C,\EEnd(E))\xra{\frac1{r(E)}Tr}\Ext^1(C,\O_S).$$
\end{proof}

We say that a morphism $p\colon X\ra S$ has \emph{a multi-section}
if there is a subscheme $Y\subset X$ such that the restriction
$p|_Y$ is a finite morphism $Y\ra S$. If such a subscheme $Y$ can be chosen to be flat over $S$ then $p$ is said to have a \emph{flat multi-section}.

\begin{corollary}
Assume that a flat morphism $p\colon X\ra
S$ of quasi-compact quasi-separated schemes over a field of characteristic zero has a flat multi-section.
Then $p$ is SCDT.
\end{corollary}
\begin{proof}
Follows from Proposition~\ref{predl_finiteflat} and
Proposition~\ref{predl_SCDTproperties}.1.
\end{proof}

\begin{predl}
Let $p\colon X\ra S$ be a smooth projective morphism of smooth schemes over a field of zero characteristic. Then $p$ is SCDT.
\end{predl}
\begin{proof}
By a result of Deligne~\cite[th. 6.1]{De_deg}, the complex $Rp_*\O_X$ is a direct sum of its cohomologies. Therefore  $p_*\O_X$ (the push-forward in the abelian category of sheaves) is a direct summand of $Rp_*\O_X$. By~\cite[th. 5.5]{De_deg},
the sheaf $p_*\O_X$ is a vector bundle. The arguments from the proof of
Proposition~\ref{predl_finiteflat} show that $\O_S$ is a direct summand of $p_*\O_X$, and hence of $Rp_*\O_X$.
\end{proof}

Now we provide an example of a flat affine (and moreover, locally trivial in Zariski topology)
morphism of smooth varieties over a field, which is not SCDT.
The derived descent category in this example is not equivalent to the derived category of the base.

\begin{example}
\label{example_gl/p} Let $V$ be a vector space over a field $\k$.
Let $X$ be the linear group $GL(V)$, let $P\subset X$ be its parabolic subgroup.
Consider the homogeneous space
$S=X/P$, it is a smooth projective variety. Denote by $d$ the dimension of $S$, denote by $p$ the quotient map $X\ra S$.
Take two line bundles $\LL_1=\O_S$ and
$\LL_2=\omega_S$ so that
$$\Hom_{\D^b_{\coh}(S)}(\LL_1,\LL_2[d])=\Ext^d(\LL_1,\LL_2)\ne 0.$$
Applying the comparison functor, we get
\begin{multline*}
\Hom_{\D^b_{\coh}(X)/p}(\Phi(\LL_1),\Phi(\LL_2[d]))\subset \Hom_{\D^b_{\coh}(X)}(p^*\LL_1,p^*\LL_2[d]) =\\
=\Hom_{\D^b_{\coh}(X)}(\O_X,\O_X[d])=0,
\end{multline*}
because $X$ is affine and $\Pic X=0$. We conclude that the comparison
functor
$\Phi\colon \D^b_{\coh}(S)\ra \D^b_{\coh}(X)/p$ is not fully faithful.
\end{example}

In the next example we demonstrate that the objects in the derived category of coherent sheaves cannot be defined locally in Zariski topology.
\begin{example}
Let $S$ be a scheme, let $S=\bigcup U_i$ be an affine covering of $S$.
Denote  $X=\bigsqcup U_i$, denote by $p\colon X\ra S$ the natural map, let
$$\Phi\colon \D^b_{\coh}(S)\ra \D^b_{\coh}(X)/p$$
be the comparison functor. For any coherent sheaf $F$ on $S$ and
$k>0$ we have
$$\Hom_{\D(X)/p}(\Phi(\O_S),\Phi(F[k]))\subset \Hom_{\D(X)}(\O_X,p^*F[k])=H^k(X,\oplus F|_{U_i})=0.$$
On the other hand, if the scheme $S$ is not affine then for some coherent sheaf
$F$ and $k>0$ one has $$\Hom_{\D(S)}(\O_S,F[k])=H^k(S,F)\ne 0.$$
In this case the functor $\Phi$ is not an equivalence and the morphism $p$
is not SCDT.
\end{example}

Now we show that the comparison functor can give  an equivalence between the bounded derived categories of coherent sheaves for a morphism   $p\colon X\ra S$ which is not SCDT.
\begin{example}
\label{example_a1}
Let $S=\A^1$ be an affine line over an algebraically closed field $\k$,
let $P_1,P_2\in \A^1$ be two different points.
Let $X=(\A^1\setminus P_1)\bigsqcup(\A^1\setminus P_2)$ be a disjoint union of two punctured lines and  $p\colon X\ra S$ be the natural mapping.
We claim that the comparison functor
$$\Phi\colon \D^b_{\coh}(S)\ra \D^b_{\coh}(X)/p$$
is an equivalence while the morphism $\O_S\ra Rp_*\O_X$ is not split and so the comparison functor
$$\D(S)\ra \D(X)/p$$
is not an equivalence.
\end{example}

\begin{proof}
The cohomological dimension of the category $\coh(\A^1)$ is $1$, hence any object in $\D^b_{\coh}(\A^1)$ is quasi-isomorphic to a direct sum of its cohomology sheaves.
Further, any coherent sheaf on $\A^1$ is a direct sum of undecomposable sheaves of the form
$$\O_{\A^1} \quad\text{or}\quad \O_{rP}, P\in\A^1,$$
where $\O_{rP}$ denotes the structure sheaf of $r$-th neighborhood of a point $P$.
We use the following notation
\begin{align*}
U_1&=\A^1\setminus P_1, \\
U_2&=\A^1\setminus P_2, \\
U_{12}&=U_1\cap U_2=\A^1\setminus \{P_1,P_2\}.
\end{align*}
To verify that $\Phi$ is fully faithful we check that
\begin{equation}
\label{equation_homhom}
\Hom_{\D^b_{\coh}(\A^1)}(H_1,H_2[k])=\Hom_{\D^b_{\coh}(X)/p}(\Phi(H_1),\Phi(H_2[k]))
\end{equation}
for any undecomposable sheaves $H_1$ and  $H_2$ and all $k$.
For
$H_1=\O_{\A^1}$ and  $k=0$ this holds because the comparison functor for abelian categories of coherent sheaves is fully faithful, see
Theorem~\ref{th_abdescentforstacks}. For $H_1=\O_{\A^1}$ and $k\ne
0$ both sides of~(\ref{equation_homhom}) vanish because $\A^1$ and $X$ are affine. Let $H_1=\O_{rP}$
and $P\ne P_1,P_2$. Then we have
\begin{multline}
\label{equation_hom4}
\Hom_{\D^b_{\coh}(\A^1)}(\O_{rP},H_2[k])=\Hom_{\D^b_{\coh}(U_{12})}(\O_{rP},H_2[k]|_{U_{12}})=\\
=\Hom_{\D^b_{\coh}(X')/p'}(\Phi'(\O_{rP}),\Phi'(H_2[k]|_{U_{12}}))=\Hom_{\D^b_{\coh}(X)/p}(\Phi(\O_{rP}),\Phi(H_2[k])),
\end{multline}
where
$X'=U_{12}\sqcup U_{12}$,
the morphism $p'\colon X'\ra U_{12}$ is the natural mapping and $\Phi'\colon
\D^b_{\coh}(U_{12})\ra \D^b_{\coh}(X')/p'$ is the comparison functor.
The middle equality in~(\ref{equation_hom4}) holds because $p'$ has a section and therefore is SCDT.
The last equality in~(\ref{equation_hom4}) is checked by comparing definitions of $\Hom$ spaces in $\D^b_{\coh}(X')/p'$ and $\D^b_{\coh}(X)/p$.
Finally, for $H_1=\O_{rP}$ and $P=P_1$ (or
$P_2$), (\ref{equation_homhom}) is also true:
$$\Hom_{\D^b_{\coh}(\A^1)}(H_1,H_2[k])=\Hom_{\D^b_{\coh}(X)}(p^*H_1,p^*H_2[k])
=\Hom_{\D^b_{\coh}(X)/p}(\Phi(H_1),\Phi(H_2[k])).$$
Here the first equality holds because $p^*H_1$ is $H_1|_{U_2}$ supported on~$U_2$ and $0$ supported on~$U_1$. The second equality holds because $H_1|_{U_{12}}=0$ and compatibility conditions are empty.

Now we check that $\Phi$ is essentially surjective.
By definition, any object of the category
$$\D^b_{\coh}(X)/p=\D^b_{\coh}(U_1\sqcup U_2)/p$$
is an object of the derived category $\D^b_{\coh}(X)$ equipped with gluing data.
Any object in $F$ in  $\D^b_{\coh}(X)$ has the form
$F=F_1\oplus F_2$, where summands
\begin{align*}
F_1=&\left(\bigoplus\limits_i \O_{U_{1}}[k_i]\right)
\bigoplus\left(\bigoplus\limits_{i: Q_i\ne
P_1} \O_{r_iQ_i}[l_i]\right), \\
F_2=&\left(\bigoplus\limits_i \O_{U_{2}}[k'_i]\right)
\bigoplus\left(\bigoplus\limits_{i: Q'_i\ne P_2}
\O_{r'_iQ'_i}[l'_i]\right)
\end{align*}
are shifts of sheaves supported on $U_1\subset X$ and $U_2\subset X$ respectively, and all sums are finite. Gluing data is an isomorphism
$$\theta\colon F_1|_{U_{12}}\xra{\sim}F_2|_{U_{12}}.$$
Since such an isomorphism exists, the sums
\begin{align*}
&\left(\bigoplus\limits_i \O_{U_{12}}[k_i]\right)
\bigoplus\left(\bigoplus\limits_{i: Q_i\ne
P_1, P_2} \O_{r_iQ_i}[l_i]\right), \\
&\left(\bigoplus\limits_i \O_{U_{12}}[k'_i]\right)
\bigoplus\left(\bigoplus\limits_{i: Q'_i\ne P_1, P_2}
\O_{r'_iQ'_i}[l'_i]\right)
\end{align*}
coincide. I.e., $F_1$ and  $F_2$ coincide everywhere except for the points $P_1$ and $P_2$. Therefore, for a certain object $H\in
\D^b_{\coh}(\A^1)$ we have $\Phi(H)=(F,\theta)$.

On the other hand, it is easily seen that the canonical mapping
$\O_S\ra Rp_*\O_X$
does not split. Indeed,
$$\Hom(Rp_*\O_X,\O_S)=\Hom(p_*\O_X,\O_S)=\Hom(\O_{\A^1\setminus P_1}\oplus \O_{\A^1\setminus P_2},\O_{\A^1})=0.$$
\end{proof}

\section{Derived descent theory for equivariant sheaves}\label{s9}

Throughout this section $X$ will denote a scheme of finite type over $\k $ (where $\k$ is an arbitrary field), by an algebraic group $G$ we will understand
a group scheme of finite type over $\k$.

Suppose an algebraic group $G$ acts on a scheme $X$. Denote by $a\colon G\times X\ra X$
the action morphism, denote by $\mu\colon G\times G\ra G$ the structure morphism of the group.
By $p_i$ and $p_{jk}$ we will denote the projections from $G\times X$ and $G\times G\times X$ onto factors.
\begin{definition}[see~\cite{BL}]
\label{def_equivariant}
By definition, a \emph{$G$-equivariant sheaf} on $X$ is a sheaf $F$ on $X$ equipped with an isomorphism
$\theta\colon p_2^*F\ra a^*F$ of sheaves on $G\times X$ satisfying the cocycle condition: on $G\times G\times X$ the following equality holds
$$(1\times a)^*\theta\circ p_{23}^*\theta=(\mu\times 1)^*\theta.$$

Morphisms of equivariant sheaves from $(F_1,\theta_1)$ to $(F_2,\theta_2)$ are by definition morphisms $f\colon F_1\ra F_2$ compatible with $\theta$, i.e. such $f$ that
$\theta_2\circ p_2^*f=a^*f\circ \theta_1$.
\end{definition}

In the interesting special case of finite groups the above definition can be reformulated as follows.
\begin{definition}
A \emph{$G$-equivariant sheaf} on $X$ is a sheaf
$F$ on $X$ equipped with isomorphisms $\theta_g\colon  F\ra g^*F$ for any
$g\in G$ such that $\theta_{gh}=h^*(\theta_g)\circ \theta_h$ for any pair
$g,h\in G$. A \emph{morphism} of $G$-equivariant sheaves from
$(F_1,\theta_{1,g})$ to $(F_2,\theta_{2,g})$ is a morphism of sheaves
$f\colon F_1\ra F_2$ such that for any $g\in G$ one has
$\theta_{2,g}\circ f=g^*f\circ \theta_{1,g}$.
\end{definition}

We denote the abelian category of $G$-equivariant quasi-coherent (resp/ coherent)
sheaves on $X$ by $\qcoh^G(X)$ (resp. by $\coh^G(X)$).
Consider the unbounded derived category of $G$-equivariant sheaves of $\O_X$-modules. Denote by $\D^G(X)$ its full subcategory of complexes with
$G$-equivariant quasi-coherent cohomology sheaves. Denote by  $\D^{b,G}_{\coh}(X)$ the full subcategory of $\D^G(X)$, formed by complexes whose cohomology sheaves are coherent and almost all equal to zero. Finally, a \emph{$G$-equivariant perfect complex} on $X$ is an object of $\D^G(X)$, which is a perfect complex on $X$ when forgetting
the equivariant structure. The  category of $G$-equivariant  perfect complexes on
$X$ is denoted by $\D^{\perf,G}(X)$.

One can look at the definition of an equivariant sheaf from a slightly different point of view. Consider the following simplicial scheme
\begin{equation}
\label{equation_X/G}
(X/G)_{\bul}=[X_0,X_1,X_2,\ldots,p_{\bul}]=[X,G\times X,G\times G\times X,\ldots,p_{\bul}],
\end{equation}
where morphisms $p_{\bul}$ are defined as follows. For a nondecreasing mapping
$f\colon [1,\ldots,m]\ra [1,\ldots,n]$ define a morphism of schemes
$$p_f\colon \underbrace{G\times \ldots\times G\times X}_{n}\ra
\underbrace{G\times \ldots\times G\times X}_{m}$$
by the rule
$$(g_n,\ldots,g_2,x_1)\mapsto (g_{f(m)}\cdot\ldots\cdot g_{f(m-1)+1},\ldots,
g_{f(2)}\cdot\ldots\cdot g_{f(1)+1},g_{f(1)}\cdot\ldots\cdot g_2x_1).$$
For small $n=m\pm 1$ and increasing $f$ the morphisms $p_{\bul}$
are shown below
$$\xymatrix{
G\times G\times X \ar@<-10mm>[rr]^-{p_{23}} \ar[rr]^-{\mu\times 1} \ar@<10mm>[rr]^-{1\times a} &&
G\times X \ar@<-5mm>[rr]^-{p_2} \ar@<5mm>[rr]^-{a}\ar@<5mm>[ll]_-{e\times 1}
\ar@<-5mm>[ll]_-{p_1\times e\times p_2}&& X \ar[ll]_-{e\times 1}}.$$
Consider the descent category
\begin{equation}
\label{equation_kernggx}
\Kern([\qcoh(X),\qcoh(G\times X),\qcoh(G\times  G\times X),\ldots,p_{\bul}^*]),
\end{equation}
associated with the cosimplicial category formed by categories of sheaves
on schemes $X_i=G^{\times i}\times X$  and pull-back functors.
A comparison of Definitions~\ref{def_dxbar} and \ref{def_equivariant} shows that the
category~(\ref{equation_kernggx}) is equivalent to the category of
$G$-equivariant quasi-coherent sheaves on~$X$, see also~\cite[6.1.2b]{De}.

For any morphism of stacks $p\colon X\ra S$ we considered in section \ref{s7} an augmented simplicial stack, formed by fibred products
$$(X\ra S)_{\bul}=[S,X,X\times_S X,X\times_S X\times_S X,\ldots,p_{\bul}].$$
It is interesting that the cosimplicial scheme~(\ref{equation_X/G}) has the same form if one takes as $p$ the canonical mapping from a scheme $X$ to the quotient stack $X\quot G$, which is the quotient of the scheme $X$ by the action of the
group~$G$. Definition and basic properties of the stack $X\quot G$
can be found in~\cite[1.3.2 and 4.14.1.1]{La}. Under hypotheses of this section $X\quot G$
is an algebraic stack of finite type over $\k$.
The morphism $p$ is strictly flat, hence by
Theorem~\ref{th_abdescentforstacks} quasi-coherent sheaves on the stack $X\quot G$ can be defined locally with respect to the covering~$p$, see also~\cite[6.23]{La}.
In the other words,
\begin{multline*}
\qcoh(X\quot G)=\Kern([\qcoh(X),\qcoh(X\times_{X\quot G}X),\qcoh(X\times_{X\quot G}X\times_{X\quot G}X),\ldots,p_{\bul}^*])=\\
=\Kern([\qcoh(X),\qcoh(G\times X),\qcoh(G\times  G\times X),\ldots,p_{\bul}^*])=
\qcoh^G(X).
\end{multline*}
That is, sheaves on $X\quot G$ are $G$-equivariant sheaves on $X$. This fact allows to use the language of stacks for working with equivariant sheaves.
Note that the categories $\D^G(X), \D^{b,G}_{\coh}(X)$ and $\D^{\perf,G}(X)$ coincide with corresponding versions of derived categories of sheaves on the stack
$X\quot G$. Also note that if the scheme $X$ (and hence, the stack $X\quot G$,
see~\cite[lemma 2.1]{AB}) is semi-separated, the category $\D^{b,G}_{\coh}(X)$ is equivalent to the bounded derived category $\D^b(\coh^G(X))$
of equivariant sheaves on~$X$.

We are interested in the following question: when an object of the derived category of sheaves on $X$ plus an action of $G$ on this object determine an object of the derived category of $G$-equivariant sheaves? To be precise, when  the category
$\D^G(X)$ is equivalent to the descent category $\Kern$,
associated with the cosimplicial category
\begin{equation}
\label{cats_DXGX}
[\D(X),\D(G\times X),\D(G\times G\times X),\ldots,Lp_{\bul}^*]?
\end{equation}
As we remarked above, the category~(\ref{cats_DXGX}) has the form
\begin{equation*}
[\D(X),\D(X\times_S X),\D(X\times_S X\times_S X),\ldots,Lp_{\bul}^*],
\end{equation*}
for the morphism of stacks $p\colon X\ra S=X\quot G$, thus the results of Section~\ref{s7} can be applied. Below we answer the question.

\medskip
Since we are going to deal with infinite dimensional representations of algebraic groups, we recall some relevant notions.
\begin{definition}
Let $G$ be an algebraic group over a field $\k$. \emph{A (rational) linear representation} of~$G$ over $\k$ can be defined in any of the following ways
(see~\cite[ch. 1, \S 1]{MF}):
\begin{enumerate}
\item as inductive limit of finite dimensional representations of $G$ over~$\k$;
\item as an object of the descent category (see Definitions~\ref{def_dxbar} and~\ref{def_dxt}) for the morphisms of stacks $p\colon\Spec\k\ra (\Spec\k)\quot G$;
\item as a $G$-equivariant sheaf on $\Spec\k$;
\item as a comodule over the coalgebra $\k[G]=\Gamma(G,\O_G)$ (this definition is suitable only for affine group schemes).
\end{enumerate}
\end{definition}

All representations in the sequel are supposed to be rational.

\begin{definition}
An algebraic group $G$ over $\k$ is called \emph{linearly reductive} (see~\cite[def. 1.4]{MF}) if the category of finite dimensional representations (or, equivalently,
of all representations) of $G$ over $\k$ is semisimple.
\end{definition}
Denote the category of all rational (resp. of finite dimensional rational)
linear representations of $G$ by $\Repr(G)$ (resp. by $\repr(G)$).

\begin{predl}
\label{prop_linred}
Let $G$ be a group scheme of finite type over $\k$.
Then the following conditions are equivalent:
\begin{enumerate}
\item $G$ is linearly reductive;
\item the natural homomorphism $\k\ra \k[G]$ is an embedding of a direct summand in the category $\Repr(G)$;
\item the comparison functor $\Phi$ is an equivalence between derived categories $\D^b(\repr(G))$ and $\D^b(\k\mmod)/p$;

{\hspace{-1cm}} If, in addition, the group scheme $G$ is affine and smooth over $\k$,
then conditions 1-3 are equivalent to the following

\item for $\har\k=l>0$: the connected component of identity $G_0\subset G$
is a diagonalizable torus $(\G_m)^r$ and the order of the finite group $G/G_0$ is co-prime to $l$; for $\har\k=0$: $G$ is reductive.
\end{enumerate}
\end{predl}
\begin{proof}
$1\Rightarrow 2$: by definition.

$2\Rightarrow 3$: follows from Theorem~\ref{th_descentforstacks} applied to
the morphism of stacks $p\colon X\ra S$ where
$X=\Spec\k$, $S=(\Spec\k)\quot G$, and $p$ is the canonical morphism.
Since the category $\qcoh(S)=\Repr(G)$ is semisimple, the sheaf $p_*\O_X=H^0(Rp_*\O_X)=\k[G]$ is a direct summand in the complex $Rp_*\O_X$.
And homomorphism $\O_S\ra p_*\O_X$ is the natural embedding of representations
$\k\ra\k[G]$, which splits by condition 2.

$2\Rightarrow 1$: Let $f\colon U\ra V$ be arbitrary embedding of representations of $G$. Let us prove that $f$ has a left inverse mapping. Consider the commutative diagram $$ \xymatrix{U \ar[r]^f \ar[d]^{\eta_U} & V\ar[d]^{\eta_V}\\
p_*p^*U \ar[r]^{p_*p^*f} & p_*p^*V}
$$
of quasi-coherent sheaves on $S$ (i.e., of representations of $G$). By a projection formula, the morphism $\eta_U$ has the form $U\otimes(\O_S\ra p_*\O_X)$ and is split by assumption.
The morphism $p^*f$ is an embedding of vector spaces. Since it is split, the morphism $p_*p^*f$ is also split. Consequently, the composition $\eta_V\circ f$ splits, and so does $f$.

$3\Rightarrow 1$: We need to prove that $\Ext^1_{\repr(G)}(V,V')=0$
for any $V,V'\in \repr(G)$. By our assumption on the comparison functor, we have
\begin{multline*}
\Ext^1_{\repr(G)}(V,V')=\Hom_{\D^b(\repr(G))}(V,V'[1])=\\
=\Hom_{\D^b(\k\mmod)/p}(\Phi(V),\Phi(V'[1]))\subset
\Hom_{\D^b(\k\mmod)}(V,V'[1])=0.
\end{multline*}

$1\Leftrightarrow 4$: this characterization is a result by M.\,Nagata~\cite{Nagata}.
\end{proof}

Introduce the following notation:
\begin{align*}
\D(X)^G&=\Kern([\D(X),\D(G\times X),\ldots,Lp_{\bul}^*]),\\
\D^{\perf}(X)^G&=\Kern([\D^{\perf}(X),\D^{\perf}(G\times X),\ldots,Lp_{\bul}^*]),\\
\D^b_{\coh}(X)^G&=\Kern([\D^b_{\coh}(X),\D^b_{\coh}(G\times X),\ldots,Lp_{\bul}^*]).\\
\end{align*}

\begin{theorem}
\label{th_descentforequiv} Let $G$ be a group scheme of finite type over a field~$\k$,
acting on a scheme $X$ of finite type over $\k$. Suppose $G$ is linearly reductive. Then the following categories are equivalent:
$$\D^G(X)\cong \D(X)^G,\qquad \D^{\perf,G}(X)\cong \D^{\perf}(X)^{G},\qquad
\D^{b,G}_{\coh}(X)\cong \D^b_{\coh}(X)^G.$$
\end{theorem}
\begin{proof}
All statements follow from Theorem~\ref{th_descentforstacks} applied to the canonical morphism of stacks $X\xra{p} X\quot G$. To verify that
$p$ is SCDT, consider the Cartesian square of stacks:
$$\xymatrix{
X \ar[d]^p \ar[r] & pt \ar[d]^t\\
X\quot G \ar[r] & pt\quot G.}$$
The morphism
$\O_{pt\quot G}\ra Rt_*t^*\O_{pt\quot G}$ is a canonical morphism $\k\ra\k[G]$
of representations of the group $G$. Since $G$ is linearly reductive, this morphism  splits, and so $t$ is SCDT. By Proposition~\ref{predl_SCDTproperties}.2,
the morphism $p$ is also SCDT.
\end{proof}

\begin{corollary}
Categories from the above theorem are equivalent if $G$ is a reductive group
over an algebraically closed field of zero characteristic or if $G$ is a finite group and the characteristic of $\k$ does not divide the order of $G$.
\end{corollary}
\begin{proof}
Under both assumptions the group $G$ is linearly reductive, see
Proposition~\ref{prop_linred}.
\end{proof}

\begin{remark} Suppose that the group $G$ is finite. By definition,
an object of the descent category $\D(X)^G$ is an object $F\in\D(X)$, equipped with isomorphisms
$\theta_g\colon F\xra{\sim} g^*F$ for any $g\in G$ such that $g^*\theta_h\circ\theta_g=\theta_{hg}$. In other words, an object of $D(X)^G$ is an object of $\D(X)$ with an action of the group $G$ on it.
\end{remark}

If the group $G$ is not linearly reductive, the statement of Theorem~\ref{th_descentforequiv} may be false.
\begin{example}
\label{example_gl/p2}
Let  $V$ be a vector space, take $X=GL(V)$ and let $G=P$ be a parabolic subgroup
in $GL(V)$, then $G$ acts on $X$ by shifts. This action is free, so there is an equivalence $\D^{b,G}_{\coh}(X)\cong\D^b(\coh^G(X))\cong \D^b(\coh(S))\cong \D^b_{\coh}(S)$
where the quotient $S=GL(V)/P$ is a homogeneous space. We are in the situation
of Example~\ref{example_gl/p}, therefore the comparison functor is not an equivalence.
\end{example}

\end{document}